\documentclass[11pt]{amsart}
\usepackage[active]{srcltx}

\usepackage[pdftex,left=2.65cm,right=2.65cm,top=2.7cm,bottom=2.7cm]{geometry}

\usepackage[colorlinks=true,urlcolor=blue,
citecolor=red,linkcolor=blue,linktocpage,pdfpagelabels,
bookmarksnumbered,bookmarksopen]{hyperref}
\usepackage{bookmark}
\usepackage{amsmath,amssymb}
\usepackage{amsthm}
\usepackage{amstext}
\usepackage{amscd}
\usepackage{enumitem} 
\usepackage{epsfig}

\usepackage{color}
\usepackage{comment}
\usepackage{esint}
\usepackage[utf8]{inputenc}
\usepackage{mathrsfs}
\usepackage{setspace}
\usepackage{mathtools}

\usepackage{lineno}

\definecolor{green}{rgb}{0,0.5,0}

\newtheorem{theorem}{Theorem}[section]

\newtheorem{proposition}{Proposition}

\newtheorem{lemma}{Lemma}[section]
\numberwithin{lemma}{section}
\newtheorem{corollary}{Corollary}
\theoremstyle{definition}
\newtheorem{remark}{Remark}[section]

\numberwithin{equation}{section}
\newtheorem{definition}{Definition}

\newcommand{\abs}[1]{\lvert #1 \rvert}

\newcommand{\norm}[1]{\lVert #1 \rVert}

\newcommand{\N}{{\mathbb N}}
\newcommand{\R}{{\mathbb R}}
\newcommand{\Z}{{\mathbb Z}}

\newcommand{\dif}{\,\mathrm{d}}

\allowdisplaybreaks

\newcommand{\un}{u_{n}}

\newcommand{\normLqplusp}[2]{\lVert #1 \rVert^{q+1}_{L^{q+1} ( #2 )}}

\newcommand{\norme}[2]{\lVert #1 \rVert_{E ( #2 )}}

\newcommand{\weakly}{\rightharpoonup}

\newcommand{\intrn}{\int_{\R^N}}

\newcommand{\half}{\frac{1}{2}}

\newcommand{\overqplus}{\frac{1}{q+1}}
\newcommand{\quarter}{\frac{1}{4}}

\newcommand{\tstar}{2^{\ast}}

\newcommand{\clambda}{c_{\lambda}}
\newcommand{\cbarlow}{\underline{c}}

\newcommand{\dx}{\text{d}x}

\begin{document}

\title[Solutions to Schr\"odinger-Poisson system]{Groundstates and infinitely many high energy solutions to a class of nonlinear Schr\"odinger-Poisson systems}

\author{Tomas Dutko}
\address{Department of Mathematics, Computational Foundry, Swansea University, Fabian Way, Swansea, U.K. SA1~8EN}
\email{662536@swansea.ac.uk}
\author{Carlo Mercuri}
\address{Department of Mathematics, Computational Foundry, Swansea University, Fabian Way, Swansea, U.K. SA1~8EN} 
\email{c.mercuri@swansea.ac.uk}
\author{Teresa Megan Tyler}
\address{Department of Mathematics, Computational Foundry, Swansea University, Fabian Way, Swansea, U.K. SA1~8EN}
\email{megan.tyler268@outlook.com}

\begin{abstract}
We study a nonlinear Schr\"{o}dinger-Poisson system which reduces to the nonlinear and nonlocal PDE  
\begin{equation*}
- \Delta u+ u + \lambda^2 \left(\frac{1}{\omega|x|^{N-2}}\star \rho u^2\right) \rho(x) u  = |u|^{q-1} u \quad x \in \R^N,
\end{equation*}
where $\omega = (N-2)\abs{\mathbb{S}^{N-1}},$ $\lambda>0,$ $q\in(1,2^{\ast} -1),$ $\rho:\R^N \to \R$ is nonnegative, locally bounded, and possibly non-radial, $N=3,4,5$ and $2^*=2N/(N-2)$ is the critical Sobolev exponent. In our setting $\rho$ is allowed as particular scenarios, to either 1) vanish on a region and be finite at infinity, or 2) be large at infinity. 
We find least energy solutions in both cases, studying the vanishing case by means of a priori integral bounds on the Palais-Smale sequences and highlighting the role of certain positive universal constants for these bounds to hold. Within the Ljusternik-Schnirelman theory we show the existence of infinitely many distinct pairs of high energy solutions, having a min-max characterisation given by means of the Krasnoselskii genus. Our results cover a range of cases where major loss of compactness phenomena may occur, due to the possible unboundedness of the Palais-Smale sequences, and to the action of the group of translations. 
\end{abstract}
\maketitle

\section{Introduction}

\noindent This paper is devoted to the nonlinear and nonlocal equation 
\begin{equation}\label{SP one equation}
\tag{$\mathcal E$}- \Delta u+ u + \lambda^2 \left(\frac{1}{\omega|x|^{N-2}}\star \rho u^2\right) \rho(x) u  = |u|^{q-1} u \quad x \in \R^N,
\end{equation}
where $\omega = (N-2)\abs{\mathbb{S}^{N-1}},$ $\lambda>0,$ $q\in(1,2^{\ast} -1),$ $\rho:\R^N \to \R$ is nonnegative, locally bounded, and possibly non-radial, $N=3,4,5$ and $2^*=2N/(N-2)$ is the critical Sobolev exponent.

We are mainly concerned with the existence and multiplicity of solutions, together with their variational characterisation. This brings us to addressing issues related to a suitable functional setting and its relevant properties, such as those related to separability and compactness. In particular, the variational formulation of \eqref{SP one equation} requires in general a functional setting different from the standard Sobolev space $H^1(\R^N).$ This is the case if the right hand side of the classical Hardy-Littlewood-Sobolev inequality
\begin{equation}\label{HLS intro}
\tag{HLS}\quad  \int_{\R^N}\int_{\R^N}\frac{u^2(x)\rho (x) u^2(y) \rho (y)}{|x-y|^{N-2}}\dif{x} \dif{y}\, \lesssim ||\rho u^2||_{L^{\frac{2N}{N+2}}(\R^N)}^2,
\end{equation}
\noindent is not finite for some $u\in H^1(\R^N).$ In what follows, we consider separately two assumptions on $\rho$:
\begin{enumerate}[label=$\mathbf{(\rho_{\arabic*})}$]
	\item \label{vanishing_rho} $\rho^{-1} (0)$ has non-empty interior and there exists $\overline{M} > 0$ such that
	\begin{equation*}
		\abs{x \in \R^N : \rho(x) \leq \overline{M}} < \infty;
	\end{equation*}
	 
	\item \label{coercive_rho} for every $M > 0$,
	\begin{equation*}
		\abs{x \in \R^N : \rho(x) \leq M} < \infty.
	\end{equation*}
\end{enumerate}
These are reminiscent of analogous assumptions considered in the `local' context of the nonlinear Schr\"odinger equation by Bartsch and Wang in \cite{Bartsch and Wang}. In particular, we will refer to \ref{vanishing_rho} as to the {\it vanishing case}, and to \ref{coercive_rho} as to the {\it coercive case}, as the latter assumption is verified if $\rho$ is locally bounded such that $\rho(x)\rightarrow \infty$ as $|x|\rightarrow \infty,$ yielding compactness properties in the functional setting which are stronger than in the other case. It is clear that \ref{vanishing_rho} is compatible with $\rho$ exploding, as well as with $\rho$ having a finite limit at infinity. The latter is a situation which yields loss of compactness phenomena to occur, in part due, in the present subcritical regime, to the action of the group of translations in $\R^N$. In this vanishing case we prove uniform a priori bounds on suitable sequences of approximated critical points, which allow us to construct nontrivial weak limits having a definite variational nature. \newline
To state and prove our results we define $E(\R^N) \subseteq H^1(\R^N)$ as
\[E(\R^N) \coloneqq \left\{u\in  W^{1,1}_{\textrm{loc}}(\R^N) \,: \, \|u\|_{E(\R^N)} < +\infty\right\},\]
\noindent with norm 
\[\|u\|_{E(\R^N)} \coloneqq \left(\int_{\R^N}(|\nabla u|^2 +  u^2)\dif x + \lambda\left(\int_{\R^N}\int_{\R^N}\frac{u^2(x)\rho (x) u^2(y) \rho (y)}{|x-y|^{N-2}}\dif x\dif y\right)^{1/2}\right)^{1/2}.\]
Variants of the space $E(\R^N)$ have been studied since the work of P.L. Lions \cite{Lions Hartree}, see e.g. \cite{RuizARMA}, and \cite{Bellazzini},\cite{Bonheure and Mercuri}, \cite{Mercuri Moroz Van Schaftingen}.  Solutions to \eqref{SP one equation} are the critical points of the $C^1(E(\R^N);\R)$ energy functional  
\begin{equation}\label{definition I}
	I_{\lambda}(u) = \frac{1}{2}\int_{\R^N}(|\nabla u|^2 + u^2)+\frac{\lambda^2}{4}\int_{\R^N}\int_{\R^N}\frac{u^2(x)\rho (x) u^2(y) \rho (y)}{\omega|x-y|^{N-2}}\dif x\dif y -\frac{1}{q+1}\int_{\R^N}|u|^{q+1}.
\end{equation} 
One could regard \eqref{SP one equation} as formally equivalent to a nonlinear Schr\"{o}dinger-Poisson system 

\begin{equation}\label{main SP system multiplicity}
\left\{
\begin{array}{lll}
- \Delta u+ u + \lambda^2 \rho (x) \phi  u = |u|^{q-1} u, \qquad &x\in \R^N,  \\
\,\,\, -\Delta \phi=\rho(x) u^2,\ &   x\in \R^N.
\end{array}
\right.
\end{equation}\\
\noindent 
In fact, it is well-known from classical potential theory that if $u^2\rho \in L^1_{\textrm{loc}
	}(\R^N)$ is such that 
\begin{equation}\label{double integral finite}
	\int_{\R^N}\int_{\R^N}\frac{u^2(x)\rho (x) u^2(y) \rho (y)}{|x-y|^{N-2}}\dif x\dif y < +\infty,
\end{equation}
\noindent then 
\begin{equation}\label{definition of phi u}
	\phi_u (x) = \int_{\R^N}\frac{\rho (y)u^2(y)}{\omega |x-y|^{N-2}}\dif y 
\end{equation}
\noindent is the unique weak solution in $D^{1,2}(\R^N)$ of the Poisson equation
\begin{equation}\label{poissoneqn}
	-\Delta \phi = \rho (x) u^2
\end{equation}
\noindent and it holds that 
\begin{equation}\label{double integral}
	\int_{\R^N} |\nabla \phi_u|^2 = \int_{\R^N} \rho \phi_u u^2  \dif x =\int_{\R^N}\int_{\R^N}\frac{u^2(x)\rho (x) u^2(y) \rho (y)}{\omega|x-y|^{N-2}}\dif{x} \dif{y} .
\end{equation}
Here we set 
	\[D^{1,2}(\R^N) = \{ u\in L^{\tstar}(\R^N) : \nabla u \in L^2(\R^N)\},\]
equipped with norm
	\[ \|u\|_{D^{1,2}(\R^N)} = \|\nabla u\|_{L^2(\R^N)}.\]
\noindent 
By elliptic regularity, the local boundedness of $\rho$ implies that any pair $(u,\phi)\in E(\R^N)\times D^{1,2}(\R^N)$ solution to (\ref{main SP system multiplicity}) is such that $u$ and $\phi$ are both of class $C^{1,\alpha}_{\textrm{loc}}(\R^N).$ In particular, if $u\geq 0$ is nontrivial, it holds that $u,\phi>0.$ Note that $\inf I_\lambda=-\infty,$ however it is an easy exercise to see that $I_\lambda$ is bounded below on the set of its nontrivial critical points by a positive constant. It therefore makes sense to define a solution $u\in E(\R^N)$ to \eqref{SP one equation} as a {\it groundstate} if it is nontrivial, and if it holds that $I_\lambda(u)\leq I_\lambda(v)$ for every nontrivial critical point $v\in E(\R^N)$ of $I_\lambda$. \newline 

Since the classical work of Ambrosetti-Rabinowitz \cite{Ambrosetti and Rabinowitz}, considerable advances have been made in the understanding of several classes of nonlinear elliptic PDE's in the absence of either the so-called Palais-Smale or the Ambrosetti-Rabinowitz conditions, yet achieving in the spirit of \cite{Ambrosetti and Rabinowitz} existence and multiplicity results; see e.g. \cite{Ambrosetti and Malchiodi2, Ambrosetti and Malchiodi, Struwe Book, Willem book}.  
In addition to those of Strauss \cite{Strauss} and Berestycki-Lions \cite{Berestycki and Lions}, which have been a breakthrough in the study of autonomous scalar field equations on the whole of $\R^N,$ a great deal of work, certainly inspired by that of Floer and Weinstein \cite{Floer and Weinstein}, has been devoted to the study of nonlinear Schr\"odinger equations with nonradial potentials and involving various classes of nonlinearities:
\begin{equation}\label{nonlinear s}
	-\Delta u + V(x) u = f(x,u), \quad x\in \R^N.
\end{equation}
The classical works of Rabinowitz \cite{Rabinowitz} and Benci-Cerami \cite{Benci Cerami} have provided a penetrating analysis on equations like \eqref{nonlinear s}, and inspired the work on various remarkable variants of it, under different hypotheses on $V$ and $f$ which may allow loss of compactness phenomena to occur. Authors have contributed to understand these phenomena in a min-max setting, in analogy to what had been discovered and highlighted in the context of minimisation problems by P.L. Lions in \cite{Lions} and related papers. An interesting case has been considered by Bartsch and Wang \cite{Bartsch and Wang} who proved existence and multiplicity of solutions to \eqref{nonlinear s} for $V(x) = 1 + \lambda^2 \rho(x),$ and with $\rho$ satisfying either \ref{vanishing_rho} or \ref{coercive_rho}. 
Years later Jeanjean and Tanaka in \cite{Jeanjean and Tanaka} and related papers, have looked into cases where $f(x,u)$ may violate the Ambrosetti-Rabinowitz condition. Remarkably, they have been able to overcome the possible unboundedness of the Palais-Smale sequences, with an approach which is reminiscent of the `monotonicity trick' introduced for a different problem by Struwe \cite{Struwe Paper}. \newline\newline
Even though our equation \eqref{SP one equation} is formally a nonlocal variant of the above nonlinear Schr\"odinger equation, there are some specific variational features that we wish to highlight, which are not shared with \eqref{nonlinear s}.  Firstly, although our nonlinearity $f(x,u)=|u|^{q-1}u$ does satisfy the Ambrosetti-Rabinowitz condition, to the best of our knowledge it is still not known whether the boundedness of the Palais-Smale sequences holds for $q\in(2,3).$ We stress that for this reason and in this range of exponents, it is not known whether the Palais-Smale condition holds, even with $\rho\equiv 1$ and working with the subspace of radial functions in $H^1(\R^N)$.
In the range $q\in(2,3],$ the relation between the mountain-pass level and the infimum over the Nehari manifold for the functional $I_\lambda$ associated to \eqref{SP one equation} seems non-straightforward; we recall that these levels coincide when dealing with \eqref{nonlinear s} for a fairly broad class of nonlinearities $f,$ see e.g. \cite[p. 73]{Willem book}.
In the case of pure power nonlinearities and $q\in(1,2],$ and unlike for the action functional associated with \eqref{nonlinear s}, the variational properties of $I_\lambda$ are particularly sensitive to $\lambda,$ yielding existence, multiplicity (of a local minimiser and at the same time of a mountain-pass solution) and nonexistence results, see e.g. \cite{RuizJFA, RuizARMA} and \cite{Mercuri}.
Finally, a natural functional setting associated to \eqref{SP one equation} may not be necessarily a Hilbert space. In fact note that assumption \ref{vanishing_rho} is compatible with a situation where $\rho(x) \rightarrow \rho_{\infty}> 0$ as $|x| \rightarrow \infty$, in which the space $E(\R^N) \simeq H^1(\R^N),$ as well as with the case $\rho(x)\rightarrow \infty$ as $|x|\rightarrow \infty,$ in which $E(\R^N)\subset H^1(\R^N);$ we tackle the case of vanishing $\rho$ with a unified approach for these particular sub-cases.
\newline

Variants of (\ref{SP one equation}) appear in the study of the quantum many--body problem, see \cite{Bao Mauser Stimming}, \cite{Catto}, \cite{Lions Hartree Fock}. The convolution term represents a repulsive interaction between particles, whereas the local nonlinearity $|u|^{q-1} u$ is a generalisation of the $u^{5/3}$ term introduced by Slater \cite{Slater} as local approximation of the exchange term in Hartree--Fock type models, see e.g. \cite{Bokanowski Lopez Soler}, \cite{Mauser}. In the last few decades, nonlocal equations like (\ref{SP one equation}) have received increasing attention on questions related to existence, non-existence, variational setting and singular limit in the presence of a parameter. We draw the reader's attention to \cite{Ambrosetti},  \cite{Benci and Fortunato},  \cite{Catto} and references therein, for a broader mathematical picture on questions related to Schr\"odinger-Poisson type systems.  Relevant contributions to the existence of positive solutions, mostly for $q>3=N,$ such as \cite{Cerami and Vaira, Cerami and Molle}, are based on the classification of positive solutions given by Kwong \cite{Kwong} to
\begin{equation*}\label{Kwong}
  - \Delta u+  u = u^{q} , \qquad x\in \R^3,  
 \end{equation*}
regarded as a `limiting' PDE when $\rho(x)\rightarrow 0,$ as $|x|\rightarrow \infty.$  Recently in \cite{Sun Wu Feng, Mercuri and Tyler}, in the case  $\rho(x)\rightarrow 1,$ as $|x|\rightarrow \infty,$ the relation between \eqref{main SP system multiplicity} and
\begin{equation}\label{A-Ruiz PDE}
\begin{cases}
	-\Delta u + u + \lambda^2 \phi u = |u|^{q-1} u, \quad &\mathbb{R}^3\\
	-\Delta \phi = u^2 &\mathbb{R}^3
\end{cases}
\end{equation}
\noindent 
as a limiting problem, has been studied, though a full understanding of the set of positive solutions to \eqref{A-Ruiz PDE} has not yet been achieved. \\

Considerably fewer results have been obtained in relation to the multiplicity of solutions. It is worth mentioning \cite{Ambrosetti and Ruiz} whose (radial) approach is suitable in the presence of constant potentials. More precisely Ambrosetti-Ruiz \cite{Ambrosetti and Ruiz} have studied the problem
\eqref{A-Ruiz PDE} with $\lambda > 0$ and $1 < q < 5$. When $q \in (1,2)\cap(3,5)$ their approach relies on the symmetric version of the Mountain-Pass Theorem \cite{Ambrosetti and Rabinowitz}, whereas for $q \in (2,3]$ and in the spirit of \cite{Jeanjean and Tanaka, Struwe Paper}, they develop a min-max approach to the multiplicity which in fact improves upon \cite{Ambrosetti and Rabinowitz} and is based on the existence of bounded Palais-Smale sequences at specific levels associated with the perturbed functional
$$
I_{\mu,\lambda}(u) = \frac{1}{2}\int_{\R^3}(|\nabla u|^2 + u^2)+ \frac{\lambda^2}{4}\int_{\R^3}\int_{\R^3}\frac{u^2(x) u^2(y)}{\omega|x-y|}\dif x\dif y - \frac{\mu}{q+1} \int_{\R^3} \abs{u}^{q+1} \dx, 
$$ for a dense set of values $\mu \in \left[\half, 1 \right).$
	
\subsection{Main Results}
In the vanishing case \ref{vanishing_rho} our main result is the following.
\begin{theorem}[{\bf Groundstates for $q\geq 3$ under under \ref{vanishing_rho}}]\label{groundstate sol} Let $N=3$, $\rho \in L^\infty_{\textrm{loc}}(\R^N)$ be nonnegative, satisfying \ref{vanishing_rho}, and $q \in [3, \tstar-1).$ There exists a positive constant $\lambda_*=\lambda_*(q,\overline{M})$ such that for every $\lambda\geq\lambda_*,$ \eqref{SP one equation} admits a positive groundstate solution $u\in E(\R^3).$ For $q>3$, $u$ is a mountain-pass solution. 
\end{theorem}
We point out that by construction  $\lambda_*= \max\{\lambda_0,\lambda_1\},$ where $\lambda_0$ and $\lambda_1$ are universal constants defined in Proposition \ref{positivityofweakuthm} and in Proposition \ref{lambdaone}, which ensure that, for every $\lambda\geq\lambda_*,$ certain Palais-Smale sequences possess weak limits with a precise variational characterisation.\newline
This result extends to a nonlocal equation that of Bartsch-Wang \cite{Bartsch and Wang}, as we are able to show in the spirit of their work that for $\lambda$ large, there are no Palais-Smale sequences at the mountain-pass level which are weakly convergent to zero, in a context where the embedding of $E(\R^3)$ into $L^{q+1}(\R^3)$ is in general non-compact. This is the case if for instance  $\rho(x)\rightarrow \rho_\infty>\overline{M},$ as $|x|\rightarrow\infty.$ In this case $E(\R^3)\simeq H^1(\R^3),$ with equivalent norms by \eqref{HLS intro}, and the non-compactness of the embedding is a well-known fact. Under \ref{vanishing_rho}, a condition `at infinity' for certain Palais-Smale sequences to be relatively compact is given in Proposition \ref{PSconditionVan} and Proposition \ref{CPSconditionVan}.\newline
It is worth observing that the arguments of Proposition \ref{lambdaone} can be adapted to the original result of Bartsch and Wang \cite[Section 5]{Bartsch and Wang} on the nonlinear Schr\"odinger equation to prove in their setting,  for the whole range of exponents and for $\lambda$ large enough, the existence of a mountain-pass solution and hence, using the Nehari characterisation of the mountain-pass level \cite[p. 73]{Willem book}, the existence of a groundstate solution. We prove Proposition \ref{lambdaone} highlighting how the `interaction' between $\lambda$ and  $\overline{M}$ appearing in \ref{vanishing_rho} yields the desired estimates. To this aim we carry out a Brezis-Lieb type splitting argument in the spirit of \cite{Brezis and Lieb}, combining it with a simple weighted $L^3$ estimate given in Lemma \ref{weightedL3boundlemma}, together with the relation between the mountain-pass level and the infimum on the Nehari manifold, which may be sensitive to whether $q=3$ or $q>3.$ \newline
To prove Theorem \ref{groundstate sol} we follow a Nehari constraint approach, paying attention to the more delicate case $q=3.$ For this exponent, it is not clear whether the mountain-pass level is critical. From a variational perspective, this is another point that makes our work different from \cite{Bartsch and Wang}; see also Theorem \ref{least energy pure homog rho} below. We stress here that \ref{vanishing_rho} may not be enough for the right continuity of the mountain-pass levels $c_\lambda(q)$ to hold as $q\rightarrow 3^+.$\newline

In the coercive case \ref{coercive_rho} we show that $E(\R^N)$ is compactly embedded in $L^p(\R^N)$ for any $2<p<2^*$ and $\lambda>0.$ This is used to prove the following.
\begin{theorem}[\textbf{Groundstates for $q\geq 3$ under \ref{coercive_rho}}]\label{existenceandleastenergy}
Let $N = 3$, $\rho \in L^\infty_{\textrm{loc}}(\R^3)$ be nonnegative, satisfying \ref{coercive_rho}, and $q \in [3,\tstar -1)$. Then, for any fixed $\lambda > 0$, \eqref{SP one equation} has both a positive mountain-pass solution and a positive groundstate solution in $E(\R^3),$ whose energy levels coincide for $q>3.$
\end{theorem}
Note that in this case the compact embedding result provided by Lemma \ref{compactembeddingofeintolpplus} allows us to have a `variationally' stronger result for $q=3,$ to be compared with Theorem \ref{groundstate sol}. Namely, we can show that the mountain-pass level is critical, using that the Palais-Smale condition is satisfied under \ref{coercive_rho} and for $3\leq q<5$. A positive mountain-pass solution, which may not be a groundstate for $q=3,$ is constructed as a strong limit of a Palais-Smale sequence living nearby the positive cone in $E(\R^3)$. \newline

When dealing with the range $q\in(2,3),$ we overcome the possible unboundedness of the Palais-Smale sequences, combining tools developed in this paper, with the approach of Jeanjean and Tanaka \cite{Jeanjean and Tanaka}. Roughly speaking, the proof is based on constructing a sequence $(u_n)_{n\in \mathbb N}$ of critical points to suitable approximated functionals
$$
	I_{n}(u) = \frac{1}{2}\int_{\R^N}(|\nabla u|^2 + u^2)+\frac{\lambda^2}{4} \int_{\R^N}\rho(x) \phi_u u^2 -\frac{\mu_n}{q+1}\int_{\R^N}|u|^{q+1},
$$
which accumulates around the desired solution when letting $\mu_n\rightarrow 1^-$ as a result of satisfying a Pohozaev-type condition stated in Lemma \ref{pohozaevlemma} (which guarantees its boundedness), and by the compactness property provided by Lemma \ref{compactembeddingofeintolpplus}. More precisely, we have the following. 

\begin{theorem}[\textbf{Groundstates for $q<3$ under \ref{coercive_rho}}]\label{mountainpasssoltheoremlowp}
Let $N=3,4,5$, $q \in (2, 3)$ if $N=3$ and $q\in(2,\tstar-1)$ if $N=4,5$. Let $\lambda>0,$ and assume $\rho \in L^\infty_{\textrm{loc}}(\R^N)\cap W^{1,1}_{loc}(\R^N)$ is nonnegative and satisfies \ref{coercive_rho}. Moreover suppose that $k\rho (x) \leq (x,\nabla \rho)$ for some $k>\frac{-2(q-2)}{(q-1)}$. Then, \eqref{SP one equation} has a mountain-pass solution $u\in E(\R^N).$  Moreover, there exists a groundstate solution.
\end{theorem}
	\begin{remark} 
	The same proof when working instead with the functional
	\begin{equation*}
	I_{+} (u) = \frac{1}{2} \int_{\R^N} \left( | \nabla u |^2 + u^2 \right) + \frac{\lambda^2}{4} \int_{\R^N}\rho(x) \phi_u u^2 - \frac{1}{q+1} \int_{\R^N} u_+^{q+1},
\end{equation*}
 allows to show that mountain-pass and groundstate critical points exist for this functional, and are positive by construction.
\end{remark}

Under \ref{coercive_rho} and for $q\leq3$, the relation between mountain-pass solutions and groundstates found in Theorem \ref{existenceandleastenergy} for $q=3$ and Theorem \ref{mountainpasssoltheoremlowp} for $q<3$ seems not obvious; in particular, it is not clear whether they actually coincide. We are able to get more insight about the variational nature of these solutions in the case $\rho$ is homogeneous of a suitable order $\bar{k}>0$, as shown in the following theorem. It is worth pointing out that this homogeneity condition is not compatible with $\rho$ vanishing on a region.
\begin{theorem}[\textbf{Homogeneous case for $q\le 3:$ mountain-pass solutions vs. groundstates}]\label{least energy pure homog rho}
Let $N=3,4,5$, $q \in (2, 3]$ if $N=3$ and $q\in(2,\tstar-1)$ if $N=4,5$. Suppose $\lambda > 0$ and $\rho \in L^\infty_{\textrm{loc}}(\R^N)\cap W^{1,1}_{loc}(\R^N)$ is nonnegative, satisfies \ref{coercive_rho}, and is homogeneous of degree $\bar{k}$, namely $ \rho(tx)=t^{\bar{k}}\rho(x)$ for all $t>0$, for some 
	$$\bar{k}>\left(\max\left\{\frac{N}{4},\frac{1}{q-1}\right\}\cdot(3-q)-1\right)_+.$$ 
Then, the mountain-pass solutions that we find in Theorem \ref{existenceandleastenergy} $(q=3)$ and Theorem \ref{mountainpasssoltheoremlowp} $(q<3)$ are groundstates.
\end{theorem}

We prove the above theorem analysing some relevant scaling properties of $I_{\lambda}$ in Proposition \ref{variational characterisation MP level low q}, which allows us to characterise the mountain-pass level in terms of the infimum over a certain manifold, defined as a suitable combination of the Nehari and Pohozaev identities. We believe that this manifold is a natural constraint.  We point the reader to Remark \ref{k bound}, in which we give an explanation of the lower bound assumption on $\bar{k}.$ 
	
In the spirit of Ambrosetti-Rabinowitz \cite{Ambrosetti and Rabinowitz} and under \ref{coercive_rho} we show that  \eqref{main SP system multiplicity}  possesses infinitely many high energy solutions. In our context it seems appropriate to distinguish the cases $q \in (3,5)$ and $q \in (2,3]$ when working within the Ljusternik-Schnirelman theory. Since for $q \in (3,5)$ Lemma \ref{compactembeddingofeintolpplus} implies that the Palais-Smale condition is satisfied, we can use the $\mathbb Z_2$-equivariant Mountain-Pass theorem, adapting to $E(\R^N)$  arguments similar to those developed for a different functional setting by Szulkin; see \cite{Szulkin}. To this aim, in Lemma \ref{separability} we prove that for $N\geq 3$ $E(\R^N)$ is a separable Banach space, by constructing a suitable linear isometry of $E(\R^N)$ onto the Cartesian product of $H^1(\R^N)$ with some of the mixed norm Lebesgue spaces studied by Benedek and Panzone \cite{Benedek Panzone}, namely $L^4(\R^N;L^2(\R^N)).$ As a consequence of this identification, we can show that $E(\R^N)$ admits a Markushevic basis, that is a set of elements $\{(e_m,e_m^*)\}_{m\in\N}\subset E(\R^N) \times E^*(\R^N)$ such that the duality product $<e_n,e_m^*>=\delta_{nm}$ for all $n,m\in\N$, the $e_m$'s are linearly dense in $E(\R^N)$, and the weak$^*$-closure of span$\{e_m^*\}_{m \in\N}$ is $E^*(\R^N)$. We use this, combined with Lemma \ref{compactembeddingofeintolpplus} to obtain lower bounds on the energy which allow us to show the divergence of a sequence of min-max critical levels defined by means of the classical notion of Krasnoselskii genus; see Lemma \ref{divergence of critical levels} below.  This yields the following 
\begin{theorem}[\textbf{Infinitely many high energy solutions for $q>3$}]\label{first multiplicity theorem}
    Let $N = 3,$ $q\in(3,\tstar -1)$ and $\lambda > 0.$ Suppose  $\rho\in L^{\infty}_{loc}(\R^3)$ is nonnegative and satisfies \ref{coercive_rho}. Then, there exist infinitely many distinct pairs of critical points $\pm u_m\in E(\R^N)$, $m\in \N$, for $I_{\lambda}$
	such that $I_{\lambda}(u_m)\to+\infty$ as $m\to+\infty$.
\end{theorem} 

When $q<3$, the above construction is not directly applicable because of the possible unboundedness of the Palais-Smale sequences. Here we use a deformation lemma due to Ambrosetti and Ruiz \cite{Ambrosetti and Ruiz}, in the flavor of the work of Jeanjean and Tanaka, which is suitable for Ljusternik-Schnirelman type results. Assuming that $\rho(x)$ is homogeneous of some order $\bar{k}>0$, allows us to define as in \cite{Ambrosetti and Ruiz}, certain classes of admissible subsets of $E(\R^N)$ and hence of min-max levels; see Lemma \ref{derivatives of f} and Lemma \ref{use of curves} below. This together with the aforementioned Pohozaev type inequality (which in the present homogeneous case becomes an identity by Euler's classical theorem) and Lemma \ref{compactembeddingofeintolpplus}, allows us to show that these min-max levels are critical, and that they are arbitrarily large, by  Lemma \ref{divergence of critical levels} again. We therefore have the following.
\begin{theorem}[\textbf{Infinitely many high energy solutions for $q\leq 3$}]\label{partial multiplicity result low q}
	Let $N=3,4,5$. Assume $q\in(2,3]$ if $N=3$ and $q\in(2,\tstar-1)$ if $N=4,5$. Suppose $\lambda > 0$ and $\rho\in L^{\infty}_{loc}(\R^N)\cap W^{1,1}_{loc}(\R^N)$ is nonnegative, satisfies \ref{coercive_rho}, and is homogeneous of degree $\bar{k}$, namely, $ \rho(tx)=t^{\bar{k}}\rho(x)$ for all $t>0$, for some 
	$$\bar{k}>\left(\max\left\{\frac{N}{4},\frac{1}{q-1}\right\}\cdot(3-q)-1\right)_+.$$ 
Then, there exist infinitely many distinct pairs of critical points, $\pm u_m\in E(\R^N)$, $m\in \N$, for $I_{\lambda}$
				such that $I_{\lambda}(u_m)\to+\infty$ as $m\to+\infty$.
\end{theorem}
\begin{remark}
	For $q=3$, the homogeneity assumption on $\rho$ is not used to prove the boundedness of the Palais-Smale sequences which holds for this exponent, but rather it is used in the construction of the min-max levels.
\end{remark}
\begin{remark}
	For $N=3,4$, we can cover all $\bar{k}>0$. For $N=5$, the threshold for $\bar{k}$ is sensitive to the range of $q$. Namely, if $q\in(\frac{11}{5},\tstar-1)$, we can cover all $\bar{k}>0$, however if $q\in(2,\frac{11}{5})$, this is not the case.
\end{remark}

\subsection{Outline}
The paper is organised as follows. In Section \ref{preliminaries section} we deal with general facts about the functional setting, we prove the separability of $E(\R^N)$ and other properties that will be used throughout, comprising positivity and regularity. We prove a Pohozaev type necessary condition that will be extensively used in the existence proofs, and that in this section is applied to a nonexistence result for $q = 2^{\ast}-1$. Here we also discuss the min-max setting and related properties, which hold for a generic locally bounded $\rho.$ 

In Section \ref{vanishing rho section} we work under the vanishing assumption  \ref{vanishing_rho}. Here we develop a set of uniform integral estimates which hold for all the values of $\lambda$ above certain lower thresholds. We conclude the section with the proof of the existence of groundstates, Theorem \ref{groundstate sol}, and provide with Proposition \ref{PSconditionVan} and Proposition \ref{CPSconditionVan}, some new compactness results on sequences of approximated critical points of $I_\lambda.$

Section \ref{coercive rho section} is devoted to the coercive case \ref{coercive_rho}. For any fixed arbitrary $\lambda>0$ we prove the compactness Lemma \ref{compactembeddingofeintolpplus}, and the existence results Theorem \ref{existenceandleastenergy}, Theorem \ref{mountainpasssoltheoremlowp}, and Theorem \ref{least energy pure homog rho}.

Section \ref{5.1} is entirely devoted to the multiplicity of high energy solutions. In particular, in Section \ref{prelims for multiplicity section} we prove Lemma \ref{divergence of critical levels} that is key to show later in the proofs the existence of a blowing up sequence of infinitely many distinct critical levels of high energy. In Section \ref{proof of multiplicity theorem section} we recall the notion of the Krasnoselskii genus and its properties, and deal with the proof of Theorem \ref{first multiplicity theorem}. Finally, in Section \ref{proof of multiplicity theorem low q section} we prove Theorem \ref{partial multiplicity result low q}.

\section*{Acknowledgements} The authors would like to thank the anonymous referees for the valuable and constructive comments, and thank them in particular for suggesting a simple proof of Lemma 2.6.

\section{Preliminaries}\label{preliminaries section}

\noindent We introduce the functional setting for our problem and provide a few preliminary lemmas that hold for all nonnegative $\rho \in L^\infty_{\textrm{loc}}(\R^N)$ and will be used under all assumptions on $\rho$. 

\subsection{Functional setting}
For what follows, we will need some properties of the functional setting which are contained in the next lemma. 
\begin{lemma}[{\bf Properties of $E(\R^N)$}]\label{separability}
Assume $N\geq 3,$ and $\rho\geq 0$ is a measurable function. The space $E(\R^N)$ is a separable Banach space that admits a Markushevic basis, that is a fundamental and total biorthogonal system, $\{(e_m,e_m^*)\}_{m\in\N}\subset E(\R^N) \times E^*(\R^N)$. Namely, $<e_n,e_m^*>=\delta_{nm}$ for all $n,m\in\N$, the $e_m$'s are linearly dense in $E(\R^N)$, and the weak$^*$-closure of span$\{e_m^*\}_{m \in\N}$ is $E^*(\R^N)$. 
\end{lemma}
\begin{proof}
Following \cite{Mercuri Moroz Van Schaftingen}, we note that we can equip $E(\R^N)$ with the equivalent norm
\begin{equation}\label{equivnorm}
\|u\|_{1}=\left(\|u\|^2_{H^1(\R^N)}+\lambda\left(\int_{\R^N}\left| I_1 \star (\sqrt{\rho}|u|)^2\right|^2\right)^{1/2}\right)^{1/2}.
\end{equation}
 Here, we have set $\alpha=1$ in $I_{\alpha}:\R^N\to\R$, the Riesz potential of order $\alpha\in(0,N)$, defined for $x\in\R^N\setminus\{0\}$ as
\[I_{\alpha}(x)=\frac{A_{\alpha}}{|x|^{N-\alpha}},\quad A_{\alpha}=\frac{\Gamma(\frac{N-\alpha}{2})}
{\Gamma(\frac{\alpha}{2})\pi^{{N}/{2}}2^{\alpha}},\]
and the choice of normalisation constant $A_{\alpha}$ ensures that the kernel $I_{\alpha}$ enjoys the semigroup property 
\begin{equation*}
I_{\alpha+\beta}=I_{\alpha}\star I_{\beta} \text{ for each } \alpha,\beta\in(0,N) \text{ such that } \alpha+\beta<N.
\end{equation*}
We first notice that the operator $T:E(\R^N)\to H^1(\R^N)\times L^4(\R^N;L^2(\R^N))$ defined by
\[(Tu)(x_0,x_1,x_2)=[u(x_0),(\lambda I_1(x_2-x_1)\rho(x_1))^{\frac{1}{2}}u(x_1)],\]
is a linear isometry from $E(\R^N)$ into the product space $H^1(\R^N)\times L^4(\R^N;L^2(\R^N)),$ endowed with the norm $$\|[u,v]\|_{\times}=\left(\|u\|^2_{H^1(\R^N)}+\|v\|_{L^4(\R^N;L^2(\R^N))}^{2}\right)^{1/2}.$$ Here $L^4(\R^N;L^2(\R^N))$ is the mixed norm Lebesgue space of  functions $v:\R^N\times\R^N\to\R$ such that
\[||v||_{L^4(\R^N;L^2(\R^N))}= \left(\int_{\R^N}\left(\int_{\R^N}|v(x_1,x_2)|^2\dif x_1\right)^2\dif x_2\right)^{{1}/{4}}<+\infty,\]
see \cite{Benedek Panzone}. Since $L^4(\R^N;L^2(\R^N))$ is a separable (see e.g. \cite[\ p.\ 107]{Pick Kufner John Fucik}) Banach space (see e.g.\ \cite{Benedek Panzone}), it follows that the linear subspace $T(E(\R^N))\subseteq H^1(\R^N)\times L^4(\R^N;L^2(\R^N)),$ and hence $E(\R^N),$ also satisfies each of these properties. Since every separable Banach space admits a Markushevic basis (see e.g. \cite{Hajek Santalucia Vanderwerff Zizler}), the proof is complete. 
\end{proof}

Reasoning as in \cite{RuizARMA} and \cite{Mercuri Moroz Van Schaftingen} it is easy to see that $C^\infty_c(\R^N)$ is dense in $E(\R^N)$ and that the unit ball in $E(\R^N)$ is weakly compact; in fact this space is uniformly convex and hence is reflexive.  The following variant to the classical Brezis-Lieb lemma will be useful to study the convergence of bounded sequences in $E(\R^N)$; see e.g. \cite{Bellazzini Frank Visciglia}, \cite{Mercuri Moroz Van Schaftingen}.
\begin{lemma}[{\bf Nonlocal Brezis-Lieb lemma}] \label{nonlocalBL}
Assume $N \geq 3$ and  $\rho(x) \in L^{\infty}_{\textrm{loc}} (\R^N)$ is nonnegative. Let $(u_n)_{n\in \mathbb N}\subset E(\R^N)$ be a bounded sequence such that $u_n \rightarrow u$ almost everywhere in $\R^N$. Then it holds that 
$$\lim_{n\rightarrow \infty} \Big[\|\nabla\phi_{u_n}\|^2_{L^2(\R^N)}-\|\nabla \phi_{(u_n-u)}\|^2_{L^2(\R^N)}\Big]=\| \nabla \phi_{u} \|^2_{L^2(\R^N)}.$$
\end{lemma}
\noindent The next simple estimate is based on an observation of  P.-L. Lions, given in \cite{ Lions Hartree Fock} for $\rho \equiv 1$; see also \cite{RuizARMA}, and \cite{Bellazzini}, \cite{Mercuri Moroz Van Schaftingen}.
\begin{lemma}[{\bf Coulomb-Sobolev inequality}]\label{weightedL3boundlemma}
Assume $N \geq 3$, $\rho(x) \in L^{\infty}_{\textrm{loc}} (\R^N)$ is nonnegative. Then the following inequality holds for all $u \in E(\R^N)$,
\begin{equation}
    \intrn \rho (x) \abs{u}^3 \leq \left( \intrn \abs{\nabla u}^2 \right)^{\half} \left( \intrn \abs{\nabla \phi_{u}}^2 \right)^{\half}.
\end{equation}
\end{lemma}

\begin{proof}
    Testing the Poisson equation \eqref{poissoneqn} with $|u|,$ the statement follows immediately by Cauchy-Schwarz inequality. 
\end{proof}

\subsection{Regularity and positivity}
Using standard elliptic regularity theory and the maximum principle, we now provide a result giving the regularity and positivity of the solutions to the Schr\"{o}dinger-Poisson system.\\

\begin{proposition}\label{reg}[{\bf Regularity and positivity}] Let $N\in[3,6]$ and $q\in[1,\tstar -1],$ $\rho\in L^\infty_{\textrm{loc}}(\R^N)$ be nonnegative and $\rho(x) \not\equiv 0$ and $(u,\phi_u)\in E(\R^N)\times D^{1,2}(\R^N)$ be a nontrivial weak solution to
\begin{equation}\label{poh system}
\left\{
\begin{array}{lll}
  - \Delta u+b u +c  \rho (x) \phi u = d|u|^{q-1} u, \qquad &x\in \R^N,  \\
  \,\,\, -\Delta \phi=\rho(x) u^2,\qquad &x\in \R^N,
\end{array}
\right.
\end{equation}
\noindent with $b, c, d \in \R_+$.
 Then, $u$, $\phi_u\in W^{2,s}_{\textrm{loc}}(\R^N),$ for every $s\geq1$, and so $u$, $\phi_u \in C^{1,\alpha}_{\textrm{loc}}(\R^N);$ moreover $\phi_u>0.$
If, in addition, $u\geq 0$, then $u>0$ everywhere.

\end{proposition}
\begin{proof}
Under the hypotheses of the proposition, both $u$ and $\phi_u$ have weak second derivatives in $L^s_{\textrm{loc}
}(\R^N)$ for all $s<+\infty$. To show this, note that from the first equation in \eqref{poh system}, we have that $-\Delta u = g(x,u)$, where 
\begin{align*}
|g(x,u)|&=|(-b u -c  \rho (x) \phi_u u + d|u|^{q-1} u| \\
&\leq C(1+|\rho \phi_u |+|u|^{q-1})(1+|u|)\\
&=: h(x)(1+|u|).
\end{align*}
Using our assumptions on $\rho$, $\phi_u$, $u$, and that $q \leq \tstar -1$, we can show that $h\in L^{N/2}_{\textrm{loc}
}(\R^N)$, which implies that $u\in L^s_{\textrm{loc}
}(\R^N)$ for all $s<+\infty$ (see e.g. \cite[p.270]{Struwe Book}). Note that here the restriction on the dimension implies that $\phi_u\in L^{N/2}_{\textrm{loc}}(\R^N).$ Since $u^2\rho \in L^s_{\textrm{loc}
}(\R^N)$ for all $s<+\infty$, then by the second equation in \eqref{poh system} and the Calder\'{o}n-Zygmund estimates, we have that $\phi_u\in W^{2,s}_{\textrm{loc}
}(\R^N)$  (see e.g.\ \cite{Gilbarg and Trudinger}). This then enables us to show that $g\in L^s_{\textrm{loc}
}(\R^N)$ for all $s<+\infty$, which implies, by Calder\'{o}n-Zygmund estimates, that $u\in W^{2,s}_{\textrm{loc}
}(\R^N)$ (see e.g.\ \cite{Gilbarg and Trudinger}). The $C^{1,\alpha}_{\textrm{loc}}(\R^N)$ regularity of both $u,\phi_u$ is a consequence of Morrey's embedding theorem. Finally, the strict positivity
is a consequence of the strong maximum principle with $L^\infty_{\textrm{loc}}(\R^N)$ coefficients \cite{Montenegro}, and this concludes the proof.
\end{proof}

\subsection{Nonexistence}
The following lemma, proved in the Appendix, will be extensively used.
\begin{lemma}\label{pohozaevlemma}[{\bf Pohozaev-type condition}]
Assume $N\in[3,6],$ $q\in[1,\tstar -1]$, $\rho \in L^\infty_{\textrm{loc}
}(\R^N) \cap W^{1,1}_{\textrm{loc}
}(\R^N)$ is nonnegative, and $k\rho(x)\leq (x,\nabla \rho)$ for some $k\in\R$. Let $(u, \phi_u) \in E(\R^N) \times D^{1,2}(\R^N)$ be a weak solution to \eqref{poh system}. Then, it holds that 
\begin{equation}\label{pohozaev}
\begin{split}
\frac{N-2}{2}\int_{\R^N}&|\nabla u|^2 \dif x +\frac{Nb}{2}\int_{\R^N}u^2\dif x  \\
&+\frac{(N+2+2k)c}{4}\int_{\R^N}\rho\phi_u  u^2 \dif x-\frac{Nd}{q+1}\int_{\R^N}|u|^{q+1}\dif x \leq 0.
\end{split}
\end{equation}
In particular the above is an identity, provided $k\rho(x)= (x,\nabla \rho)$ (by Euler's theorem, this is the case if $\rho$ is homogeneous of order $k,$ see e.g. \cite[p. 296]{Gelfand and Shilov}). 
\end{lemma}
Although we will use the above necessary condition mainly for existence purposes,  this also allows us to find a family of nonexistence results in a certain range of the parameters $N,q,\lambda,k$.
\begin{proposition}[{\bf Nonexistence: the critical case $q=2^*-1$}]\label{nonexistence for critical q}
Assume $N\in[3,6],$ $q=\tstar -1$, $\rho \in L^\infty_{\textrm{loc}
}(\R^N) \cap W^{1,1}_{\textrm{loc}
}(\R^N)$ nonnegative, $k\rho(x)\leq (x,\nabla \rho)$ for some $k\geq\frac{N-6}{2},$ and $\lambda> 0$. Let $(u, \phi_u) \in E(\R^N)\times D^{1,2}(\R^N)$ be a weak solution to \eqref{main SP system multiplicity}. Then, $(u, \phi_u)=(0,0).$
\end{proposition}
\begin{proof}
Combining the Nehari identity $I_\lambda'(u)(u)=0$ with Lemma \ref{pohozaevlemma} yields
\begin{equation*}
\Big(\frac{N-2}{2}-\frac{N}{q+1}\Big)\int_{\R^N}|\nabla u|^2 \dif x +\Big(\frac{N}{2}-\frac{N-2}{2}\Big)\int_{\R^N}u^2\dif x +\Big(\frac{2k+6-N}{4}\Big)\lambda^2\int_{\R^N}\rho\phi_u  u^2 \dif x \leq 0.
\end{equation*}
Hence, $$\int_{\R^N}u^2\dif x \leq 0,$$ and this concludes the proof.
\end{proof}

\begin{remark}
Similar nonexistence results have been obtained in the case of constant potentials and for $N=3$, in \cite{D'Aprile and Mugnai}. We point out that the in the above proposition $\lambda>0$ is arbitrary and the condition on $\rho$ is compatible with \ref{vanishing_rho}, as well as with \ref{coercive_rho}. It is interesting to note that for $N=6$ we have  $q=2^*-1=2,$ namely nonexistence occurs in a `low-$q$' regime, under both conditions \ref{vanishing_rho} and \ref{coercive_rho}. The proof shows also that for supercritical exponents $q+1>2^*$ and higher dimensions, under further regularity assumptions required for Lemma \ref{pohozaevlemma} to hold, nonexistence also occurs. 
\end{remark}

\begin{proposition}[{\bf Nonexistence: the case $q\in(1,2]$}]\label{nonexistence for lowest q}
Assume $N\geq3,$ $q\in(1,2]$, $\rho \in L^\infty_{\textrm{loc}
}(\R^N)$ and $\rho(x)\geq 1$ almost everywhere and $\lambda\geq\frac{1}{2}$. Let $u\in E(\R^N) \cap L^{q+1}(\R^N)$ satisfy  
\begin{equation}\label{nonex-eq}
- \Delta u+ u + \lambda^2 \left(\frac{1}{\omega|x|^{N-2}}\star \rho u^2\right) \rho(x) u  = |u|^{q-1} u, \qquad \textrm{in}\,\,\mathcal D'( \mathbb R^N). 
\end{equation}
Then, $u\equiv 0.$
\end{proposition}
 We note that this proposition is stated to cover also the dimensions $N> 2\left(\frac{q+1}{q-1}\right),$ namely the supercritical cases $3\geq q+1>2^*$ where $E(\R^N)$ does not embed in $L^{q+1}(\R^N).$ 
\begin{proof}
 By density we can test (\ref{nonex-eq}) by $u$ and so we obtain
\begin{equation}\label{nonexistenceforlowpequation}
    \intrn \abs{\nabla u}^2 + u^2 + \lambda^2 \rho (x) \phi_u u^2 - \abs{u}^{q+1} = 0.
\end{equation}
Following \cite[Theorem 4.1]{RuizJFA}, 
by Lemma \ref{weightedL3boundlemma} and Young's inequality we have
\begin{equation}\label{poissonon}
    \intrn \rho(x) \abs{u}^3 \leq  \intrn \abs{\nabla u}^2 + \frac{1}{4} \int_{\R^N}\rho (x) \phi_u u^2.
\end{equation}
Combining \eqref{nonexistenceforlowpequation} and \eqref{poissonon}, we have for all $\lambda \geq \half$
\begin{equation*}
  0\geq \intrn u^2 + \rho(x) \abs{u}^3 - \abs{u}^{q+1}\geq\int_{\R^N}f(u),
\end{equation*}
where $f(u) = u^2 +  \abs{u}^3 - \abs{u}^{q+1}$ is positive except at zero. Hence $u \equiv 0,$ and this concludes the proof.
\end{proof}

\subsection{Min-max setting}
The present section is devoted to the min-max properties of $I_{\lambda},$ which will be used in our existence results.

\begin{lemma}[\bf{Mountain-Pass Geometry for $I_{\lambda}$}]\label{mpgfori}
	Assume $N = 3,4,5$, $\rho(x) \in L^\infty_{\textrm{loc}}(\R^N)$ is nonnegative and $q \in (2, 2^* -1]$. Then, it holds that
	\begin{enumerate}[label=(\roman*)]
		\item $I_{\lambda} (0) = 0$ and there exist constants $r,a > 0$ such that $I_{\lambda}(u) \geq a$ if $\norme{u}{\R^N} = r;$
		\item there exist $v \in E(\R^N)$ with $\norme{v}{\R^N} > r$ such that $I_{\lambda}(v) \leq 0$.
	\end{enumerate}	
\end{lemma}

\begin{proof}
Statement $(i)$ follows reasoning as in Lemma \ref{nonzeroweaklimitlemma}. To show $(ii)$, pick $u\in C^{1}(\R^N)$, supported in the unit ball, $B_1$. Setting $v_t(x) \coloneqq t^2u(tx)$ we find that 
\begin{equation}\label{rescaledI}
\begin{split}
I_{\lambda} (v_t) =\frac{t^{6-N}}{2} \int_{\R^N} &|\nabla u|^2 + \frac{t^{4-N}}{2} \int_{\R^N} u^2 + \frac{ t^{6-N}}{4} \lambda^2\int_{\R^N}\int_{\R^N} \frac{ u^2(y)\rho(\frac{y}{t})u^2(x)\rho(\frac{x}{t})}{\omega |x-y|^{N-2}}\dif y\,\dif x \\
&- \frac{ t^{(2q+2-N)}}{q+1} \int_{\R^N}|u|^{q+1}.
\end{split}
\end{equation}
Since for every $t\geq 1$ and for almost every $x\in B_1$ we have $\rho(x/t)\leq ||\rho||_{L^{\infty}(B_1)},$ the fact that  $2q+2>6$ in \eqref{rescaledI} yields $I_{\lambda} (v_{t}) \to -\infty$ as $t\to +\infty,$ and this is enough to conclude the proof.
\end{proof}

To prove our results for $q<3$, we will need to work with a perturbed functional, $I_{\mu,\lambda}:E(\R^N)\to\R^N$, defined by
\begin{equation}\label{definition I mu lambda}
I_{\mu,\lambda}(u) = \frac{1}{2}\int_{\R^N}(|\nabla u|^2 + u^2)+\frac{\lambda^2}{4}\int_{\R^N}  \rho\phi_u u^2 -\frac{\mu}{q+1}\int_{\R^N}|u|^{q+1},\quad \mu\in\left[\frac{1}{2}, 1\right].
\end{equation}
As in Lemma \ref{mpgfori}, $I_{\mu,\lambda}$ has the mountain-pass geometry in $E(\R^N)$ for all $\mu\in\left[\frac{1}{2},1\right]$. This, as well as the monotonicity of $I_{\mu,\lambda}$ with respect to $\mu$, imply that we can define
the min-max level associated with $I_{\mu,\lambda}$ as
\begin{equation}\label{mp level low q}
c_{\mu,\lambda}= \inf_{\gamma\in\Gamma_{\lambda}}\max_{t\in[0,1]} I_{\mu,\lambda} (\gamma(t)),\quad \mu\in\left[\frac{1}{2},1\right]
\end{equation}
where 
\begin{equation}\label{gamma lambda}
{\Gamma}_{\lambda} = \{ \gamma \in C([0,1], E(\R^N)) : \gamma (0) =0, \, I_{\frac{1}{2},\lambda}(\gamma(1))< 0\}.
\end{equation}
Since the mapping $[1/2,1] \ni \mu \mapsto c_{\mu,\lambda}$ is non-increasing and left-continuous in $\mu$ (see  \cite[Lemma $2.2$]{Ambrosetti and Ruiz}) and the non-perturbed functional $I_{\lambda}$ has the mountain-pass geometry by Lemma \ref{mpgfori}, we are now in position to define the min-max level associated with $I_{\lambda}$ for all $q\in(2,\tstar-1)$.
\begin{definition}[\bf{Definition of mountain-pass level for $I_{\lambda}$}]
We set
\begin{equation}\label{mountainpasslevelforI}
{c}_{\lambda}= \left\{
\begin{array}{lll}
  c_{1,\lambda},\  &q\in(2,3),  \\
   \inf\limits_{\gamma \in \bar{\Gamma}_{\lambda}} \max\limits_{t \in [0,1]} \ I_{\lambda}(\gamma (t)),\ &  q\in[3,\tstar-1),
\end{array}
\right.
\end{equation}
where $c_{1,\lambda}$ is given by \eqref{mp level low q} and $\bar{\Gamma}_{\lambda}$ is the family of paths defined as
\begin{equation}\label{bar gamma lambda}
	\bar{\Gamma}_{\lambda} = \left\{ \gamma \in C([0,1]; E(\R^N)) : \gamma(0) = 0, I_{\lambda}(\gamma(1)) < 0 \right\}.
\end{equation}
\end{definition}
\color{black}

The remainder of this subsection is devoted to further characterisations of the min-max level $c_{\lambda}$ for $q\leq3$. We first require the following technical lemma.

\begin{lemma}\label{derivatives of f}
Suppose $N\geq 3$, $q>2$ and $\nu>\max\left\{\frac{N}{2},\frac{2}{q-1}\right\}$. Let $\bar{k}\in\left(\frac{\nu(3-q)-2}{2},\frac{4\nu-N-2}{2} \right)$. Define $f:\R^+_0\to \R$ as
\[f(t)=a{t^{2\nu+2-N}}+bt^{2\nu-N}+c t^{4\nu-N-2-2\bar{k}}-d t^{\nu(q+1)-N}, \quad t\geq0,\]
where $a, b, c, d\in\R$ are such that $a, b, d>0$, $c\geq0$. Then, $f$ has a unique critical point corresponding to its maximum.
\end{lemma} 
\begin{remark}\label{k bound}
We point out that our range of parameters ensures that $f(t)\to -\infty$ as $t\to +\infty$ and it holds that $$\left(\frac{\nu(3-q)-2}{2},\frac{4\nu-N-2}{2} \right)\bigcap \left(\frac{(\nu+1)(3-q)-2}{2},\frac{4(\nu+1)-N-2}{2} \right)\neq \emptyset.$$ In Theorem \ref{least energy pure homog rho} and Theorem \ref{partial multiplicity result low q}, we use Lemma \ref{derivatives of f}, assuming
$$ 
\bar{k}> \max\left\{\frac{N}{4},\frac{1}{q-1}\right\}(3-q)-1
$$
for $\bar{k}$ to belong to one of these intervals.
\end{remark}
\begin{proof}[Proof of Lemma \ref{derivatives of f}]
Note that by our assumptions, we can write
$$f(t)=\sum_{i=1}^k a_it^{p_i}-t^p,$$
where $a_i\geq0$, $0\leq p_i<p$ and both $a_i$, $p_i\neq0$ for some $i$. Setting $s=t^p$, we find
$$f(s)=\sum_{i=1}^k a_is^{\frac{p_i}{p}}-s.$$
It follows that $f(s)$ is strictly concave and has a unique critical point, which is a maximum. Since our assumptions ensure that $f(t)\to -\infty$ as $t\to +\infty,$ we can conclude.
\end{proof}

\noindent To state our next result, for any $\nu\in\R$, we set
\begin{equation}\label{definition of M lambda}
\bar{\mathcal{M}}_{\lambda,\nu}= \left\{u\in E(\R^N)\setminus \{0\}:J_{\lambda,\nu}(u)=0\right\},
\end{equation}
where $J_{\lambda,\nu}:E(\R^N)\to\R^N$ is defined as
\begin{equation}\label{definition of J lambda}
    \begin{split}
      J_{\lambda,\nu}(u)=\frac{2\nu+2-N}{2}\int_{\R^N}&|\nabla u|^2+\frac{2\nu-N}{2}\int_{\R^N} u^2\\
      &+\frac{4\nu-N-2-2\bar{k}}{4}\cdot \lambda^2\int_{\R^N}\rho\phi_u u^2- \frac{\nu(q+1)-N}{q+1}\int_{\R^N}| u|^{q+1}. 
    \end{split}
\end{equation}
Notice that, if $\rho$ is homogeneous of order $\bar{k},$ $J_{\lambda,\nu}(u)$ is the derivative of the polynomial $f(t)=I_{\lambda}(t^{\nu}u(t\cdot))$ at $t=1$.
\begin{proposition}[\bf{Mountain-pass characterisation of groundstates}]\label{variational characterisation MP level low q}
Let $N=3,4,5$, $q \in (2, 3]$ if $N=3$ and $q\in(2,\tstar-1)$ if $N=4,5$. Suppose $\rho \in L^\infty_{\textrm{loc}}(\R^N)\cap W^{1,1}_{loc}(\R^N)$ is nonnegative and is homogeneous of degree $\bar{k}$, namely $ \rho(tx)=t^{\bar{k}}\rho(x)$ for all $t>0$, for some 
\[\bar{k}>\max\left\{\frac{N}{4},\frac{1}{q-1}\right\}(3-q)-1.\]
Then, there exists $\nu>\max\{\frac{N}{2},\frac{2}{q-1}\}$ such that $${c}_{\lambda}=\inf_{u\in \bar{\mathcal{M}}_{\lambda,\nu}}I_{\lambda}(u)=\inf_{u\in E(\R^N)\setminus\{0\}}\max_{t\geq0} I_{\lambda}(t^{\nu} u(t\cdot)),$$
where ${c}_{\lambda}$ and $\bar{\mathcal{M}}_{\lambda,\nu}$ are defined in \eqref{mountainpasslevelforI} and \eqref{definition of M lambda}, respectively.
\end{proposition}

\begin{proof}
We first note that under the assumptions on the parameters, it holds that
\[\frac{4\nu-N-2}{2}>\frac{(\nu+1)(3-q)-2}{2}.\]
It follows from this and the lower bound assumption on $\bar{k}$ that we can always find at least one interval 
\[\left(\frac{\nu(3-q)-2}{2},\frac{4\nu-N-2}{2}\right), \quad \text{with } \nu>\max\left\{\frac{N}{2},\frac{2}{q-1}\right\},\]
that contains $\bar{k}$. We fix $\nu$ corresponding to such an interval. 
We break the remainder of the proof into a series of claims.\\

\noindent \textbf{Claim 1.} $\inf_{u\in E(\R^N)\setminus\{0\}}\max_{t\geq0} I_{\lambda}(t^{\nu} u(t\cdot))\leq \inf_{u\in \bar{\mathcal{M}}_{\lambda,\nu}}I_{\lambda}(u)$\\

\noindent To see this, let $u\in E(\R^N)\setminus \{0\}$ be fixed and consider the function 
\begin{equation}\label{definition of g least energy}
    \begin{split}
    g(t)&= I_{\lambda}(t^{\nu} u(t\cdot))\\
    &=a{t^{2\nu+2-N}}+bt^{2\nu-N}+c t^{4\nu-N-2-2\bar{k}}-d t^{\nu(q+1)-N}, \quad t\geq0,
\end{split} 
\end{equation}
where \[a=\frac{1}{2} \int_{\R^N}|\nabla u|^2, \,\, b=\frac{1}{2} \int_{\R^N} u^2,\,\, c=\frac{\lambda^2}{4} \int_{\R^N}\rho\phi_u u^2, \,\, d=\frac{1}{q+1} \int_{\R^N}|u|^{q+1}.\]
By Lemma \ref{derivatives of f}, it holds that $g$ has a unique critical point, $t=\tau_u$, corresponding to its maximum. Moreover, we can see that
\begin{align*}
    g'(t)&=\frac{\dif I_{\lambda}(t^{\nu} u(t\cdot))}{\dif t}\\
    &=\frac{2\nu+2-N}{2}\cdot t^{2\nu+1-N}\int_{\R^N}|\nabla u|^2+\frac{2\nu-N}{2}\cdot t^{2\nu-N-1} \int_{\R^N} u^2\\
      &\qquad +\frac{4\nu-N-2-2\bar{k}}{4}\cdot t^{4\nu-N-3-2\bar{k}}\cdot \lambda^2\int_{\R^N}\rho\phi_u u^2- \frac{\nu(q+1)-N}{q+1}\cdot t^{\nu(q+1)-N-1}\int_{\R^N}| u|^{q+1},  
\end{align*}
and so
\[g'(t)=0 \iff t^{\nu} u(t\cdot)\in \bar{\mathcal{M}}_{\lambda,\nu}.\]
Taken together, we have shown that for any $u\in E(\R^N)\setminus\{0\}$, there exists a unique $t=\tau_u$ such that $\tau_u^{\nu} u({\tau_u}\cdot)\in \bar{\mathcal{M}}_{\lambda,\nu}$ and the maximum of $I_{\lambda}(t^{\nu} u(t\cdot))$ for $t\geq0$ is achieved at $\tau_u$. Thus, it holds that
\begin{equation*}
    \begin{split}
    \inf_{u\in E(\R^N)\setminus\{0\}}\max_{t\geq0} I_{\lambda}(t^{\nu} u(t\cdot))&\leq \max_{t\geq0} I_{\lambda}(t^{\nu} u(t\cdot))
    =I_{\lambda}(\tau_u^{\nu} u({\tau_u}\cdot)), \quad \forall u\in E(\R^N)\setminus\{0\},    
    \end{split}
\end{equation*}
from which we can deduce that the claim holds.\\

\noindent \textbf{Claim 2.} ${c}_{\lambda} \leq \inf_{u\in E(\R^N)\setminus\{0\}}\max_{t\geq0} I_{\lambda}(t^{\nu}u(t\cdot)).$\\

\noindent By the assumptions on our parameters, we can deduce that $\nu(q+1)-N>2\nu+2-N$ and $\nu(q+1)-N>4\nu-N-2-2\bar{k}$. It follows that $I_{\lambda}(t^{\nu} u(t\cdot))< 0$ for every $u\in E(\R^N)\setminus \{0\}$ and $t$ large. Similarly, $I_{\frac{1}{2},\lambda}(t^{\nu} u(t\cdot))< 0$ for every $u\in E(\R^N)\setminus \{0\}$ and $t$ large. Therefore, we obtain
\[{c}_{\lambda} \leq \max_{t\geq0} I_{\lambda}(t^{\nu} u(t\cdot)),\quad \forall u\in E(\R^N)\setminus\{0\},\]
and the claim follows.\\

\noindent \textbf{Claim 3.} $\inf_{u\in \bar{\mathcal{M}}_{\lambda,\nu}}I_{\lambda}(u)\leq {c}_{\lambda}.$\\

\noindent We define
\[A_{\lambda,\nu}=\left\{u\in E(\R^N)\setminus \{0\}: J_{\lambda,\nu}(u) > 0\right\}\cup \{0\},\]
and first note that $A_{\lambda,\nu}$ contains a small ball around the origin. Indeed, arguing as in the proof of Lemma \ref{nonzeroweaklimitlemma}, we can show that for every $u\in E(\R^N)\setminus \{0\}$ and any $\beta>0$, we have
\begin{equation*}\label{localmin least energy low q}
\begin{split}
J_{\lambda,\nu} (u) \geq \frac{2\nu-N}{2}&||u||_{H^1(\R^N)}^2- \left(\frac{4\nu-N-2-2\bar{k}}{\omega}\right)\left(\frac{\beta-1}{4}\right)||u||^4_{H^1(\R^N)}\\
&+\left(\frac{4\nu-N-2-2\bar{k}}{\omega}\right)\left(\frac{\beta-1}{4\beta}\right) ||u||_{E(\R^N)}^4 -\frac{S_{q+1}^{-(q+1)}(\nu(q+1)-N)}{q+1}||u||_{H^1(\R^N)}^{q+1}.
\end{split}
\end{equation*}
\noindent We now pick $\delta=\left(\frac{(2\nu-N)(q+1)S_{q+1}^{q+1}}{4(\nu(q+1)-N)}\right)^{1/(q-1)}$ and note that since $\nu>\frac{N}{2}$, it follows that $\delta>0$. We assume $||u||_{E(\R^N)}<\delta$ and choosing $\beta>1$ sufficiently near $1$ we obtain
\begin{align*}
J_{\lambda,\nu} (u) &\geq \left[ \frac{2\nu - N}{4}-\left(\frac{4\nu-N-2-2\bar{k}}{\omega}\right)\left( \frac{\beta-1}{4} \right) \delta^2\right] ||u||_{H^1(\R^N)}^2 +\left(\frac{4\nu-N-2-2\bar{k}}{\omega}\right)\left(\frac{\beta-1}{4\beta}\right) ||u||_{E(\R^N)}^4\\
&\geq \left(\frac{4\nu-N-2-2\bar{k}}{\omega}\right)\left(\frac{\beta-1}{4\beta}\right) ||u||_{E(\R^N)}^4,
\end{align*}
which is strictly positive by our choice of $\nu$. This is enough to prove that $A_{\lambda,\nu}$ contains a small ball around the origin. Now, notice that if ${u}\in A_{\lambda,\nu}$, then $g'(1)>0$, where $g$ is defined in \eqref{definition of g least energy}. Since $g(0)=0$ and we showed in Claim $1$ that $\tau_u$ is the unique critical point of $g$ corresponding to its maximum, it follows that $1<\tau_u$. Using the facts that $I_{\lambda}(0)=0$ and $g'(t)=\frac{\dif I_{\lambda}(t^{\nu} u(t\cdot))}{\dif t}\geq0$ for all $t\in[0, \tau_u]$, we obtain that $I_{\lambda}(t^{\nu} u(t\cdot))\geq0$ for all $t\in[0, \tau_u]$ and, in particular, at $t=1$. Thus, we have shown $I_{\lambda}(u)\geq 0$, which also implies that $I_{\frac{1}{2},\lambda}(u)\geq 0$, for every $u\in A_{\lambda,\nu}$. Therefore, every $\gamma \in \Gamma_{\lambda}$ and every $\gamma \in \bar{\Gamma}_{\lambda}$, where $\Gamma_{\lambda}$ and $\bar{\Gamma}_{\lambda}$ are given by \eqref{gamma lambda} and \eqref{bar gamma lambda} respectively, has to cross $\bar{\mathcal{M}}_{\lambda,\nu}$, and so the claim holds. \\

\noindent \textbf{Conclusion.} Putting the claims together, it is clear that the statement holds.
\end{proof}

\subsection{Palais-Smale sequences}
We recall that a sequence $(\un)_{n\in\N} \subset E(\R^N)$ is said to be a Palais-Smale sequence for $I_{\lambda}$ at some level $c\in\R$ if
\begin{equation*}
    I(\un) \rightarrow c, \quad I'(\un) \rightarrow 0, \quad \text{as} \ n \rightarrow \infty.
\end{equation*}
If any such a sequence is relatively compact in the $E(\R^N)$ topology, then we say that the functional $I_{\lambda}$ satisfies the Palais-Smale condition at level $c$.

\begin{lemma}[\bf{Boundedness of Palais-Smale sequences}]\label{suff conds bounded PS large q}
	Assume $N = 3,4$, $\rho\in L^\infty_{\textrm{loc}}(\R^N)$ is nonnegative, $q\in[3,\tstar-1]$, and $(u_n)_{n\in\N} \subset E(\R^N)$ is a Palais-Smale sequence for $I_{\lambda}$ at any level $c>0$. Then, for any fixed $\lambda > 0$, $(u_n)_{n\in\N}$ is bounded in $E(\R^N)$.
\end{lemma}
\noindent We stress that our assumption on $N$ yields $3\leq 2^*-1.$ 
\begin{proof}
For convenience, set $$a_n= ||u_n||_{H^1(\R^N)},\qquad b_n=\lambda\left( \intrn\phi_{u_n}u_n^2\rho(x)\right)^{\frac{1}{2}},\qquad c_q=\min{\left\{\left(\frac{q-1}{2}\right),\left(\frac{q-3}{4}\right)\right\}}$$
and note that, as $n\to+\infty,$

	\begin{equation}\label{whole equation bdd PS}
	C_1+o(1)||u_n||_{E(\R^N)}\geq (q+1)I_{\lambda}(u_n)-I_{\lambda}'(u_n)(u_n)=\left(\frac{q-1}{2}\right)a^2_n+ \left(\frac{q-3}{4}\right)b^2_n
	\end{equation}
	for some $C_1>0$. Assuming $||u_n||_{E(\R^N)}\to+\infty$, we show a contradiction in each of the cases:
	
	\begin{enumerate}[label=(\roman*)]
		\item $a_n$, $b_n\to +\infty$,
		\item $a_n$ bounded and $b_n \to+\infty$,
		\item $a_n\to+\infty$ and $b_n$ bounded.
	\end{enumerate}
		First consider $q>3$. If $b_n\to+\infty$, for large $n$ we have $b_n^2\geq b_n$ and by (\ref{whole equation bdd PS}) we get
	$$C_1+o(1)||u_n||_{E(\R^N)}\geq c_q||u_n||_{E(\R^N)}^2,\,n\to+\infty,
$$
a contradiction in case $(\textrm{i})$ and $(\textrm{ii})$. If $a_n\to+\infty$ and $b_n$ is bounded, then $||u_n||_{E(\R^N)}\sim a_n,$
hence
	$$
	C_1+o(1)a_n\geq c_q a_n^2,\,\,n\to+\infty,
	$$
	a contradiction in case $(\textrm{iii})$. This makes the proof complete for $q>3$.\\
	Consider now $q=3$. By Sobolev inequality we have
$$
	C_2 \geq I_{\lambda}(u_n)\geq \frac{1}{2} a_n^2+\frac{1}{4} b_n^2 -C_3 a_n^{4},
$$
for some $C_2$, $C_3>0,$
	which yields a contradiction in case $(\textrm{ii})$. On the other hand if $a_n\to+\infty$, from the same estimate we have
	\begin{equation}\label{lessim ineq}
	b_n\lesssim a_n^2,\,n\to+\infty.
	\end{equation}
	Note that \eqref{whole equation bdd PS} yields
	\begin{equation}\label{chain of inequality bdd PS 2}
	C_1+o(1)||u_n||_{E(\R^N)}\geq a_n^2,\,\,n\to+\infty.
	\end{equation} 
Dividing by $||u_n||_{E(\R^N)}=\left(a_n^2+b_n\right)^{\frac{1}{2}}$, we get
$\frac{a_n^4}{a_n^2+b_n}= o(1),\,n\to+\infty,$
	hence
	\[b_n\gtrsim a_n^4,\,\,n\to+\infty,\]
	a contradiction in case $(\textrm{iii})$. This and \eqref{lessim ineq}, give
	\[a_n^4\lesssim a_n^2,\,\,n\to+\infty,\]
	a contradiction in case $(\textrm{i})$. This completes the proof.
\end{proof}

\begin{lemma}[\bf{Lower bound uniform in $\lambda$ for PS sequences at level $c_\lambda$}]\label{nonzeroweaklimitlemma}
Assume $N= 3,4,5$, $\lambda > 0$, $q \in (2,\tstar -1]$, $\rho \in L^\infty_{\textrm{loc}}(\R^N)$ is nonnegative. There exists a universal constant $\alpha=\alpha(q) > 0$ independent of $\lambda$ such that for any Palais-Smale sequence $(\un)_{n \in \N}\subset E(\R^N)$ for $I_{\lambda}$ at level $c_{\lambda},$ it holds that
	\begin{equation*}
	\liminf_{n \rightarrow \infty} \normLqplusp{\un}{\R^N} \geq \alpha.
	\end{equation*}
\end{lemma}
\begin{proof}
	For every $u\in E(\R^N),$ denoting $S_{q+1}$ the best constant such that $S_{q+1}\|u\|_{L^{q+1}(\mathbb R^N)}\leq \| u\|_{H^1(\mathbb R^N)},$ we have 
\begin{align*}
I_{\lambda} (u) \geq \frac{1}{2}||u||_{H^1(\R^N)}^2+\frac{\lambda^2}{4}\int_{\R^3}\rho \phi_u u^2  -\frac{S_{q+1}^{-(q+1)}}{q+1}||u||_{H^1(\R^N)}^{q+1}.
\end{align*}
\noindent Since $\omega\lambda^2\int_{\R^N} \rho \phi_u u^2 = \left(||u||^2_{E(\R^N)} - ||u||_{H^1(\R^N)}^2\right)^2,$ estimating the term $||u||^2_{E(\R^N)} ||u||_{H^1(\R^N)}^2$ with Young's inequality, we have for any $\beta>0$
\begin{align}\label{localmin}
I_{\lambda} (u) \geq \frac{1}{2}||u||_{H^1(\R^N)}^2- \frac{1}{\omega}\left(\frac{\beta-1}{4}\right)||u||^4_{H^1(\R^N)}+\frac{1}{\omega}\left(\frac{\beta-1}{4\beta}\right) ||u||_{E(\R^N)}^4 -\frac{S_{q+1}^{-(q+1)}}{q+1}||u||_{H^1(\R^N)}^{q+1}.\nonumber
\end{align}
\noindent We now pick $\delta=\left(\frac{(q+1)S_{q+1}^{q+1}}{4}\right)^{1/(q-1)}$ and assume $||u||_{E(\R^N)}<\delta$, which also implies that $||u||_{H^1(\R^N)}<\delta$. Then, choosing $\beta>1$ sufficiently near $1$ we obtain
\begin{align*}
I_{\lambda} (u) &\geq \left[ \frac{1}{4}-\frac{1}{\omega}\left( \frac{\beta-1}{4} \right) \delta^2\right] ||u||_{H^1(\R^N)}^2 +\frac{1}{\omega}\left(\frac{\beta-1}{4\beta}\right) ||u||_{E(\R^N)}^4\\
&\geq \frac{1}{\omega}\left(\frac{\beta-1}{4\beta}\right) ||u||_{E(\R^N)}^4.
\end{align*}
 We note here that both $\delta$ and $\beta$ depend on $q$ but not on $\lambda.$ Thus, we have shown that if $||u||_{E(\R^N)}=\delta/2$, then $I_{\lambda}(u)\geq\cbarlow$, for some $\cbarlow>0$ independent of $\lambda.$  So, since every path connecting the origin to where the functional $I_{\lambda}$ is negative crosses the sphere of radius $\delta/2$, it follows that $$\clambda \geq \cbarlow \text{ for every } \lambda \geq 0.$$ 
	For convenience, set $$a_n= ||u_n||_{H^1(\R^N)},\qquad b^2_n=\lambda\left( \intrn\phi_{u_n}u_n^2\rho(x)\right)^{\frac{1}{2}},$$
where $(u_n)_{n\in \mathbb N}$ is an arbitrary Palais-Smale sequence at the level $c_\lambda.$ It holds that
	\begin{align*}
	\clambda +o(1)-\|I_\lambda'(u_n)\|_{E'(\R^N)}\|u_n\|_{E(\R^N)} &\leq  I_{\lambda} (\un) -  I^{\prime}_{\lambda} (\un)  \un \\
	&= \left(\frac{1}{2}-1\right)a_n^2+\left(\frac{1}{4}-1\right)b_n^4+\left(1-\frac{1}{q+1}\right)\|u_n\|_{q+1}^{q+1}.
	\end{align*}
By concavity note that $\|u_n\|_{E(\R^N)}\leq a_n+b_n,$ hence the above yields 
	$$\cbarlow +o(1)\underbrace{-\|I_\lambda'(u_n)\|_{E'(\R^N)}(a_n+b_n)+\frac{1}{2}\left(a_n^2+b_n^4\right)}_{c_n}\leq\|u_n\|_{q+1}^{q+1},$$
	and it is easy to see that $\liminf c_n\geq0.$ The conclusion follows then with $\alpha:=\cbarlow.$

	\color{black}
\end{proof}

\section{The case of $\rho$ vanishing on a region}\label{vanishing rho section}	
\noindent Throughout this section we will make the assumption that
\begin{enumerate}[label=$\mathbf{(\rho_{\arabic*})}$]
	\item \label{vanishing_rho} $\rho^{-1} (0)$ has non-empty interior and there exists $\overline{M} > 0$ such that
	\begin{equation*}
		\abs{x \in \R^N : \rho(x) \leq \overline{M}} < \infty.
	\end{equation*}
\end{enumerate}

\noindent In what follows it is convenient to set 
	\begin{equation*}
		A(R) = \{ x \in \R^N : \abs{x} > R, \ \rho (x) \geq \overline{M} \},
	\end{equation*}
	\begin{equation*}
		B(R) = \{ x \in \R^N : \abs{x} > R, \ \rho (x) < \overline{M} \},
	\end{equation*}
for any $R>0$.
\begin{lemma}[\bf{Key vanishing property}]\label{measureofBgoingtozerolemma}
Suppose $\rho$ is a measurable function and that for some $\overline{M} \in \R$ it holds that
$$
\overline{B}:=\abs{x \in \R^N : \rho(x) < \overline{M}} < \infty.
$$
Then $$\lim\limits_{R\rightarrow\infty}|B(R)|=0.$$
\end{lemma}
\begin{proof}
The conclusion follows by the dominated convergence theorem as $B(R)\subseteq \overline{B}$ yields $$|B(R)|=\int_{\overline{B}}\chi_{B(R)}(x)\dif x\leq |\overline{B}|.$$ 
\end{proof}

\begin{lemma}[\bf{Uniform bounds in $\lambda$ for PS sequences at level $c_\lambda$}]\label{boundedPSforvanishingrho}
	Assume $N = 3,4$, $\rho \in L^\infty_{\textrm{loc}}(\R^N)$ is nonnegative, satisfying \ref{vanishing_rho}, $q\in[3,\tstar-1]$, $\lambda > 0.$ There exists a universal constant $\overline{C}=\overline{C}(q,N)>0$ independent of $\lambda,$  such that for any Palais-Smale sequence $(u_n)_{n\in\N} \subset E(\R^N)$ for $I_{\lambda}$ at level $c_{\lambda}$ it holds that $\|u_n\|_{E(\R^N)}<\overline{C}.$
\end{lemma}

\begin{proof}
	Let $v \in C^{\infty}_c (\R^N)\setminus \{0\}$ have support in $\rho^{-1} (0)$. Pick $t_v>0$ such that $I_0(t_v v)<0$ and set $v_t=tt_v v.$ Then, by definition of $\clambda$,
	\begin{equation}\label{uniform upper bound mp level}
		\clambda \leq \max_{t \in[0,1]} I_{\lambda} (v_t)=\max_{t\geq 0}I_0(tv)=:\overline{c}\,\,\footnote{In fact this bound holds in dimensions $N= 3,4,5$ and every $q\in(2,2^*-1].$}.
	\end{equation}
	Note that since $(u_n)$ is bounded by Lemma \ref{suff conds bounded PS large q}, it holds that
	\begin{align*}
	\clambda &= \lim_{n \rightarrow \infty} ( I_{\lambda} (\un) - \frac{1}{q+1} I^{\prime}_{\lambda} (\un) \cdot \un )\\
	&= \lim_{n \rightarrow \infty} \Big( \Big( \frac{1}{2} - \frac{1}{q+1} \Big) \|\un\|^2_{H^1(\mathbb R^N)}+\lambda^2 \Big( \frac{1}{4} - \frac{1}{q+1} \Big) \int_{\R^N}\phi_{\un} \rho(x) \un^2 \Big).
	\end{align*}
	The conclusion follows immediately in the case $q>3$.
	For $q=3$ the above yields a uniform bound independent on $\lambda$ for the $H^1(\R^N)$ norm and hence for the $L^{q+1}(\R^N)$ norm as well by Sobolev's inequality. Since
	$$\lambda^2\limsup_{n \rightarrow\infty}   \int_{\R^N} \phi_{\un} \un^2 \rho(x)\leq 4 \left(c_\lambda+\limsup_{n \rightarrow\infty}\left( \|u_n\|^2_{H^1(\R^N)}+\|u_n\|^{q+1}_{L^{q+1}(\R^N)}\right)\right),$$
	this concludes the proof.
\end{proof}

\begin{lemma}[\bf{Control on the tails of uniformly bounded sequences}]\label{intoutsideballlemma}
	Assume $N = 3,4,5,$ $\rho \in L^{\infty}_{\textrm{loc}}(\R^N)$ is nonnegative,  satisfying \ref{vanishing_rho}, and $(\un)_{n \in \N}\subset E(\R^N)$ is bounded uniformly with respect to $\lambda$. Then, for every $\beta > 0$ there exists $\lambda_{\beta} > 0$ and $R_{\beta} > 0$ such that for $\lambda > \lambda_{\beta}$ and $R > R_{\beta}$,
	\begin{equation*}
		||{\un}||_{L^{3}(\R^N \setminus B_R)}^{3} < \beta.
	\end{equation*}
\end{lemma}

\begin{proof}
By Lemma \ref{weightedL3boundlemma} we have
\begin{equation}
\lambda \int_{\R^N} \rho(x) \abs{\un}^3  \leq C \|u_n\|_{E(\R^N)}^3 \leq C', 
\end{equation}
for some positive constant $C'$ independent of $\lambda.$ Hence
$$
		\int_{A(R)} \abs{\un}^3 
		\leq \frac{C'}{\lambda \overline{M}} .
$$
Also observe that by H\"older's inequality and Lemma \ref{measureofBgoingtozerolemma} we have
	\begin{align*}
		\int_{B(R)} \abs{\un}^3 &\leq \Big( \int_{\R^N} \abs{\un}^{2^*} \Big)^{\frac{3}{2^*}} \Big( \int_{B(R)} 1 \Big)^{\frac{2^*-3}{2^*}}\\
		&\leq C'' \norme{\un}{\R^N}^3 \cdot \abs{B(R)}^{\frac{2^*-3}{2^*}}\\
		&\leq C''' \abs{B(R)}^{\frac{2^*-3}{2^*}}\rightarrow 0 .
	\end{align*}
	as $R\rightarrow \infty,$ again for some uniform constant $C'''>0.$ Note that our assumption on $N$ yields $3<2^*.$ This is enough to conclude the proof.
\end{proof}

\begin{proposition}[{\bf Nonzero weak limits of PS sequences at level $c_\lambda$ for $\lambda$ large}]\label{positivityofweakuthm}
    Let $N=3$, $\rho \in L^\infty_{\textrm{loc}}(\R^N)$ be nonnegative, satisfying \ref{vanishing_rho}, and $q \in [3, 5).$  There exist universal positive constants $\lambda_0=\lambda_0(q,\overline{M})$ and $\alpha_0=\alpha_0(q),$ such that if for some $\lambda\geq\lambda_0,$ $u\in E(\R^3)$ is the weak limit of a Palais-Smale sequence for $I_{\lambda}$ at level $c_{\lambda},$ then it holds that
    \begin{equation*}
         \int_{ \R^3}|u|^3\dif x > \alpha_0 .
    \end{equation*}
\end{proposition}

\begin{proof} 
Let $(u_n)_{n\in\mathbb N}\subset E(\R^3)$ be an arbitrary Palais-Smale sequence at level $c_\lambda.$ Note that we can pick $\alpha(q) > 0$ independent of $\lambda$ and of the sequence such that
	\begin{equation*}
\liminf_{n \rightarrow \infty}\|u_n\|^3_{L^3(\R^3)} \geq \alpha(q).
	\end{equation*}
	Indeed by interpolation 
	$$\int_{\R^3}|u_n|^{q+1}\leq\Big(\int_{\R^3}|u_n|^3\Big)^{\frac{5-q}{3}}\Big(\int_{\R^3}|u_n|^6\Big)^{\frac{q-2}{3}}$$
	and the claim follows by Sobolev inequality and the uniform bound given by Lemma \ref{boundedPSforvanishingrho} and by Lemma \ref{nonzeroweaklimitlemma}. In particular, recall that by Lemma \ref{boundedPSforvanishingrho}, there exists a universal constant $\overline{C}=\overline{C}(q,N)>0$ independent of $\lambda$ and of the sequence, such that $\|u_n\|_{E(\R^N)}<\overline{C}.$ By Lemma \ref{intoutsideballlemma}, it follows than that we can pick $\lambda_0(q,\overline{M})$  and  $R_\alpha>0$ such that such that for every $\lambda\geq \lambda_0$ and every $R>R_\alpha$ we have $$\limsup_{n\rightarrow\infty} \norm{\un}^{3}_{L^{3} (\R^3 \setminus B_R)} < \frac{\alpha}{2}.$$  By the classical Rellich theorem, passing if necessary to a subsequence, we can assume that $u_n\rightarrow u$ in $L^{3}_{\textrm{loc}}(\mathbb R^3).$ Therefore, for every $R>R_\alpha$, we have
    \begin{equation*}
     \norm{u}^{3}_{L^{3} (B_R)} =\lim_{n\rightarrow \infty}\norm{u_n}^{3}_{L^{3} (B_R)}\geq \liminf\limits_{n\rightarrow \infty}\norm{\un}^{3}_{L^{3} (\mathbb R^3)}  -\limsup\limits_{n\rightarrow \infty}\norm{\un}^{3}_{L^{3} (\R^3 \setminus B_R)}   > \frac{\alpha}{2}.
    \end{equation*}
    The conclusion follows with $\alpha_0=\alpha/2.$
\end{proof}

\begin{proposition}[\bf{Energy estimates for $\lambda$ large}]\label{lambdaone} Let $N=3$, $\rho \in L^\infty_{\textrm{loc}}(\R^3)$ be nonnegative, satisfying  \ref{vanishing_rho}, and $q \in [3, 5).$  Let $\lambda_0$ be defined as in Proposition \ref{positivityofweakuthm}.  There exists a universal constant $\lambda_1=\lambda_1(q,\overline{M})>0$ such that, if $\lambda\geq\max\left(\lambda_0,\lambda_1\right)$ and $u$ is the nontrivial weak limit in $E(\R^3)$ of some Palais-Smale sequence $(u_n)_{n\in\mathbb N}\subset E(\mathbb R^3)$ for $I_\lambda$ at level $c_\lambda,$ then it holds that
\begin{itemize}
\item $I_\lambda(u)=  c_\lambda, \quad \textrm{for}\,\,q\in(3,5),$
\item []
\item $\inf_{v\in\mathcal N_\lambda}I_\lambda(v)\leq I_{\lambda}(u)\leq c_\lambda,\quad \textrm{for}\,\,q=3.$
\item []
\end{itemize}
In particular, for all $\lambda\geq\max\left(\lambda_0,\lambda_1\right),$ the mountain-pass level $c_\lambda$ is critical for $q\in(3,5),$ as well as the level $I_\lambda(u)$ for $q=3.$  
\end{proposition}

\begin{proof}
By Proposition \ref{positivityofweakuthm}, for every $q\in[3,\tstar-1)$ and $\lambda\geq\lambda_0$, passing if necessary to a subsequence, we can assume that $u_n\rightharpoonup u\in E(\R^3)\setminus\{0\}$ weakly in $E(\R^3)$ and almost everywhere, for some Palais-Smale sequence $(u_n)_{n\in\mathbb N}\subset E(\mathbb R^3)$ for $I_\lambda$ at level $c_\lambda.$ By a standard argument $u$ is a critical point of $I_\lambda.$ For sake of clarity we break the proof into two steps.\\
\textbf{Step 1:} We first show that there exists a universal constant $C=C(q)>0$ such that for every $\lambda\geq \lambda_0,$ $R>0$ and $n\in\mathbb N,$ it holds that
\begin{equation}
\begin{aligned}\label{boundforerrorterm}
    I_{\lambda}(\un - u) &\geq \left( \frac{1}{4} - S_{\lambda} S^{-1} \left( \int_{A(R)} \abs{\un - u}^6 \right)^{\frac{2}{3}} \right) \int_{\R^3} \abs{\nabla(\un - u)}^2\\
    & \qquad \qquad - C \abs{B(R)}^{\frac{5-q}{6}} - \overqplus \int_{\abs{x} < R} \abs{\un - u}^{q+1},
\end{aligned}
\end{equation}
where
\begin{equation*}
    S_{\lambda} := (q-2)[3(q+1)]^{\frac{-3}{q-2}}\left( \frac{2(5-q)}{\lambda \overline{M}} \right)^{\frac{5-q}{q-2}},
\end{equation*}
$S=3(\pi/2)^{4/3}$ is the Sobolev constant, and $\overline{M}$ is defined as in \ref{vanishing_rho}. Reasoning as in Lemma \ref{weightedL3boundlemma} and by Lemma \ref{intoutsideballlemma} we obtain, 
\begin{align}\label{initialboundforerrorterm}
    \nonumber I_{\lambda}(\un - u) &\geq \quarter \int_{\R^3} \abs{\nabla (\un - u)}^2 + \quarter \int_{\R^3} \abs{\nabla (\un - u)}^2\\ 
    \nonumber&\qquad \qquad + \frac{\lambda^2}{4} \int_{\R^3} \phi_{(\un - u)} (\un - u)^2 \rho (x) - \overqplus \int_{\R^3} \abs{\un - u}^{q+1}\\
    \nonumber&\geq \quarter \int_{\R^3} \abs{\nabla (\un - u)}^2 + \frac{\lambda}{2} \int_{\R^3} \rho (x) \abs{\un - u}^3 - \overqplus \int_{\R^3} \abs{\un - u}^{q+1}\\
    &\geq \quarter \int_{\R^3} \abs{\nabla (\un - u)}^2 + \frac{\lambda \overline{M}}{2} \int_{A(R)} \abs{\un - u}^3 - \overqplus \int_{\R^3} \abs{\un - u}^{q+1}.
\end{align}
 Note that
\begin{equation*}
     \int_{\R^3} \abs{\un - u}^{q+1} =  \int_{\abs{x} < R} ... +  \int_{A(R)}... + \int_{B(R)} ...
\end{equation*}
Using that $(u_n)_{n\in\mathbb N}$ is uniformly bounded in $E(\R^3)$ and arguing as in  Lemma \ref{intoutsideballlemma} and by Sobolev inequality, we have
\begin{equation}\label{errortermBbound}
    \int_{B(R)} \abs{\un - u}^{q+1} \leq C_1 \norm{\un - u}_{L^6 (\R^3)}^{q+1} \abs{B(R)}^{\frac{5 - q}{6}} \leq C_2 \abs{B(R)}^{\frac{5-q}{6}}. 
\end{equation}
By the interpolation and Young's inequalities we obtain for all $\delta>0$ that
\begin{align*}
    \overqplus \int_{A(R)} \abs{\un - u}^{q+1} &\leq \overqplus \left( \int_{A(R)} \abs{\un - u}^3 \right)^{\frac{5-q}{3}} \left( \int_{A(R)} \abs{\un - u}^6 \right)^{\frac{q-2}{3}}\\
    &\leq \left( \frac{5-q}{3} \right) \left( \frac{\delta}{q+1} \right)^{\frac{3}{5-q}} \int_{A(R)} \abs{\un - u}^3 + \left( \frac{q-2}{3} \right) \delta^{\frac{-3}{q-2}} \int_{A(R)} \abs{\un - u}^6.
\end{align*}
In particular, we can set
\begin{equation*}
    \delta = \left( \frac{\lambda \overline{M}}{2} \cdot \frac{3}{5-q} \right)^{\frac{5-q}{3}} (q+1).
\end{equation*}
Hence
\begin{align}\label{errortermAbound}
    \nonumber \overqplus \int_{A(R)} \abs{\un - u}^{q+1} &\leq \frac{\lambda \overline{M}}{2} \int_{A(R)} \abs{\un - u}^3 + S_{\lambda} \int_{A(R)} \abs{\un - u}^6\\
    &\leq \frac{\lambda \overline{M}}{2} \int_{A(R)} \abs{\un - u}^3 + S_{\lambda} S^{-1} \left( \int_{A(R)} \abs{\un - u}^6 \right)^{\frac{2}{3}} \int_{\R^3} \abs{\nabla (\un - u)}^2,
\end{align}
where we have used Sobolev's inequality written as
\begin{equation*}
    S \left( \int_{A(R)} \abs{\un - u}^6 \right)^{\frac{1}{3}} \leq \int_{\R^3} \abs{\nabla(\un - u)}^2.
\end{equation*}
Putting together \eqref{initialboundforerrorterm},  \eqref{errortermBbound} and \eqref{errortermAbound}, the claim \eqref{boundforerrorterm} follows.\\
\textbf{Step 2: Conclusion.}  By the classical Brezis-Lieb lemma and Lemma \ref{nonlocalBL} we have 
\begin{equation}\label{BL decompositions}
c_\lambda=\lim_{n\rightarrow\infty}I_\lambda(u_n)=I_\lambda(u)+\lim_{n\rightarrow\infty}I_\lambda(u_n-u).
\end{equation}
Note that there exists a positive constant $\lambda_1=\lambda_1(q,\overline{M})$ such that for every $\lambda\geq \lambda_1$ it holds that
\begin{equation}\label{positivity of error term}
    \quarter - S_{\lambda} S^{-3} \overline{C}^4 \geq 0,
\end{equation}
where $\overline{C}$ is defined via Lemma \ref{boundedPSforvanishingrho} by the property $\|u_n\|_{E(\R^3)}<\overline{C} .$ Note that, again by the Brezis-Lieb lemma, we have 
\begin{equation*}
   \int_{A(R)} \abs{\un - u}^6  = \int_{A(R)} \abs{\un}^6 - \int_{A(R)} \abs{u}^6 + o_n (R),
\end{equation*}
with $\lim_{n\rightarrow \infty} o_n (R)=0$ for any fixed $R>0$; since by Sobolev's inequality it holds that
\begin{equation*}
   \int_{A(R)} \abs{\un}^6 \leq S^{-3} \left( \int_{\R^3} \abs{\nabla \un}^2 \right)^3 \leq S^{-3} \overline{C}^6,
\end{equation*}
we obtain the estimate
\begin{equation}\label{BL6}
    \limsup_{R \rightarrow \infty} \limsup_{n \rightarrow \infty} \int_{A(R)} \abs{\un - u}^6 \leq  S^{-3} \overline{C}^6.
\end{equation}
We conclude, by \eqref{boundforerrorterm}, \eqref{positivity of error term}, \eqref{BL6} and the classical Rellich theorem that
\begin{align*}
    \lim\limits_{n \rightarrow \infty} I_{\lambda} (\un - u) &\geq \liminf\limits_{R \rightarrow \infty} \liminf\limits_{n \rightarrow \infty} \left( \frac{1}{4} - S_{\lambda} S^{-1} \left( \int_{A(R)} \abs{\un - u}^6 \right)^{\frac{2}{3}} \right) \int_{\R^3} \abs{\nabla(\un - u)}^2\\
    &\geq \left[ \quarter - S_{\lambda} S^{-3} \overline{C}^4 \right] \liminf\limits_{n \rightarrow \infty} \int_{\R^3} \abs{\nabla (\un - u)}^2 \geq 0,
\end{align*}
and hence by \eqref{BL decompositions} that $I_{\lambda} (u) \leq \clambda$. On the other hand, since $u \in \mathcal{N}_{\lambda}$, it holds that
\begin{equation*}
    \inf\limits_{v \in \mathcal{N}_{\lambda}} I_{\lambda} (v) \leq I_{\lambda} (u) \leq \clambda,
\end{equation*}
and this completes the proof for $q=3$. For $q\in(3,\tstar-1)$, since $$\clambda = \inf\limits_{v \in \mathcal{N}_{\lambda}} I_{\lambda} (v),$$ it follows that $I_{\lambda} (u) = \clambda,$ and this concludes the proof.
\end{proof}

\begin{remark}[{\bf On the Palais-Smale condition}]\label{PSremark}
When $q>3,$ the fact that $\lim I_\lambda(u_n-u)=0$ for $\lambda$ large suggests that the Palais-Smale condition at the mountain-pass level $c_\lambda$ can be recovered in some cases. To illustrate this, note that the assumption  \ref{vanishing_rho} is compatible with having, say $\rho(x)\rightarrow 2\overline{M},$ as $|x|\rightarrow\infty,$ namely a situation where lack of compactness phenomena may occur for the system \eqref{main SP system multiplicity} as a consequence of the invariance by translations of \eqref{A-Ruiz PDE}, which plays the role of a `problem at infinity', see e.g. \cite{Mercuri and Tyler}. We stress here that $\rho$ may approach its limit from below as well as from above. To see that in this case the Palais-Smale condition is satisfied for $\lambda$ large, denote by $I_{\lambda}^{\rho\equiv2\overline{M}}$ the functional associated to \eqref{SP one equation} with $\rho\equiv2\overline{M},$ and observe that in this situation $E(\R^3)\simeq H^1(\R^3),$ with equivalent norms by \eqref{HLS intro}. Reasoning as in  \cite[Proposition 1.6]{Mercuri and Tyler}, there exist $l\in \N\cup\{0\}$, functions $(v_1,\ldots, v_l)\subset H^1(\R^3)$, and sequences of points $(y_n^j)_{n\in\N}\subset \R^3$, $1\leq j \leq l$, such that, passing if necessary to  a subsequence,
\begin{itemize}
\item $v_j$ are possibly nontrivial critical points of $I_{\lambda}^{\rho\equiv2\overline{M}}$ for  $1\leq j\leq l$,
\item []
\item $| y_n^j | \to +\infty$, $|y_n^j - y_n^{j'}| \to +\infty$ as $n\to+\infty$ if $j\neq j'$,
\item []
\item $||u_n-u- \sum_{j=1}^{l} v_j(\cdot -y_n^j)||_{H^1(\R^3)}\to 0$ as $n\to+\infty$,
\item []
\item $c_\lambda = I_{\lambda}(u)+\sum_{j=1}^{l}I_{\lambda}^{\rho\equiv2\overline{M}}(v_j)$.\item []
\end{itemize}
It is standard to see that $I_{\lambda}^{\rho\equiv2\overline{M}}$ is uniformly bounded below on the set of its nontrivial critical points by a positive constant, independent on $\lambda$. It then follows that for all $\lambda \geq\max\left(\lambda_0,\lambda_1\right),$ Proposition \ref{lambdaone} and the above yield $c_\lambda=I_\lambda(u)$ and at the same time $l=0;$ as a consequence the Palais-Smale condition is satisfied at the level $c_\lambda.$ These considerations yield the following
\begin{proposition}[{\bf Palais-Smale condition under \ref{vanishing_rho}}]\label{PSconditionVan}
Let $N=3<q$ and $\rho\geq 0$ be locally bounded such that \ref{vanishing_rho} is satisfied and such that $\rho(x)\rightarrow \rho_\infty>\overline{M}$ as $|x|\rightarrow \infty$. Let $\lambda_0$ and $\lambda_1$
be as in Proposition \ref{lambdaone}. Then, for all $\lambda\geq\max\left(\lambda_0,\lambda_1\right),$ $I_\lambda$ satisfies the Palais-Smale condition at the mountain-pass level $c_\lambda.$
\end{proposition}

It is not obvious how to prove the above proposition in the case $q=3;$ nevertheless the same considerations on strong convergence apply instead to approximated critical points of $I_\lambda$ constrained on the Nehari manifold, see the proof Theorem \ref{groundstate sol} and Proposition \ref{CPSconditionVan} below. 
\end{remark}
\subsection{Proof of Theorem \ref{groundstate sol}}
Now that we have the necessary preliminaries we present the proof of Theorem \ref{groundstate sol}.

\begin{proof}[Proof of Theorem \ref{groundstate sol}]
We recall that 
\[\mathcal{N}_{\lambda}\coloneqq\left\{u\in E(\R^3)\setminus\{0\}:G_{\lambda}(u)=0\right\},\]
where
\[G_{\lambda}(u)=I_{\lambda}'(u)(u)=||u||_{H^1(\R^3)}^2+\lambda^2\int_{\R^3}\rho\phi_u u^2-||u||_{L^{q+1}(\R^3)}^{q+1}.\]
We note that for all $q\in[3,\tstar-1)$, it is standard to see that $\mathcal{N}_{\lambda}$ is nonempty.  Moreover, we claim that the conditions
\begin{enumerate}[label=(\roman*)]
\item $\exists r>0:B_r\cap \mathcal{N}_{\lambda}=\emptyset$,
\item $G_{\lambda}'(u)(u)\neq 0,\quad \forall u\in\mathcal{N}_{\lambda}$,
\end{enumerate}
are satisfied, and so, by standard arguments, it follows that the Nehari manifold $\mathcal{N}_{\lambda}$ is a natural constraint (see e.g.\ \cite{Ambrosetti and Malchiodi}). Indeed, for $(i)$, we notice that if $u\in\mathcal{N}_{\lambda}$, then
\[0=||u||_{H^1(\R^3)}^2+\lambda^2\int_{\R^3}\rho\phi_u u^2-||u||_{L^{q+1}(\R^3)}^{q+1}\geq ||u||_{H^1(\R^3)}^2-S_{q+1}^{-(q+1)}||u||_{H^1(\R^3)}^{q+1}, \]
from which it follows that
\begin{equation}\label{lower bound nehari}
||u||_{E(\R^3)}\geq ||u||_{H^1(\R^3)} \geq S_{q+1}^{(q+1)/(q-1)},\quad\forall u\in\mathcal{N}_{\lambda}.
\end{equation}
Setting $r=S_{q+1}^{(q+1)/(q-1)}-\delta$ for some small $\delta>0$ yields $(i)$. For $(ii)$, we notice that if $u\in\mathcal{N}_{\lambda}$, then by the definition of the Nehari manifold, the assumption $q\geq3$ and \eqref{lower bound nehari}, it holds that
\begin{equation}\label{G' less that zero}
\begin{split}
    G_{\lambda}'(u)(u)&=2||u||^2_{H^1(\R^3)}+4\lambda^2\int_{\R^3}\rho\phi_u u^2-(q+1)||u||_{L^{q+1}(\R^3)}^{q+1}\\
    &=(1-q) ||u||^2_{H^1(\R^3)}+(3-q)\lambda^2\int_{\R^3}\rho\phi_u u^2\\
    &\leq (1-q) S_{q+1}^{2(q+1)/(q-1)}\\
    &<0.
    \end{split}
\end{equation}
Thus, the claim holds and so the Nehari manifold is a natural constraint. Now, if $q\in(3,\tstar-1)$, setting $\lambda_*= \max\{\lambda_0,\lambda_1\}$, the conclusion follows immediately from Proposition \ref{lambdaone} and the following characterisation of the mountain-pass level, $$\clambda = \inf\limits_{v \in \mathcal{N}_{\lambda}} I_{\lambda} (v).$$ On the other hand, assume $q=3$ and $\lambda\geq\max\{\lambda_0,\lambda_1\}$. We note that $$c_{\lambda}^*\coloneqq\inf\limits_{v \in \mathcal{N}_{\lambda}} I_{\lambda} (v)$$ is well-defined since $\mathcal{N}_{\lambda}$ is nonempty, and so, we take $(\tilde{w}_n)_{n\in\N}\subset \mathcal{N}_{\lambda}$ to be a minimising sequence for $I_{\lambda}$ on $\mathcal{N}_{\lambda}$, namely, $I_{\lambda}(\tilde{w}_n)\to c_{\lambda}^*$. By the Ekeland variational principle (see e.g.\ \cite{Costa}), there exists another minimising sequence $({w}_n)_{n\in\N}\subset \mathcal{N}_{\lambda}$ and $\xi_n\in\R$ such that 
\begin{equation}\label{Ekeland 1}
    I_{\lambda}({w}_n)\to c_{\lambda}^*,
\end{equation}
\begin{equation}\label{Ekeland 2}
    I_{\lambda}'({w}_n)({w}_n)=0,
\end{equation}

and
\begin{equation}\label{Ekeland 3}
    I_{\lambda}'({w}_n)-\xi_nG_{\lambda}'({w}_n)\to 0,\qquad \textrm{in}\,\,(E(\R^3))'.
\end{equation}
Now, by Proposition \ref{lambdaone}, \eqref{uniform upper bound mp level}, \eqref{Ekeland 1} and \eqref{Ekeland 2}, it holds that
\[\lim_{n\to+\infty}\left(I_{\lambda}({w}_n)-\frac{1}{q+1}I_{\lambda}'({w}_n)({w}_n)\right)=c_{\lambda}^*\leq c_{\lambda}\leq \bar{c},\]
for some $\bar{c}$ independent of $\lambda$.
We can therefore argue as in Lemma \ref{boundedPSforvanishingrho} to show that \begin{equation}\label{unif bound tilde w}
||{w}_n||_{E(\R^3)}< \bar{C},
\end{equation}
where $\bar{C}>0$ is the same constant independent of $\lambda$ given by Lemma \ref{boundedPSforvanishingrho}. Moreover, since $({w}_n)_{n\in\N}\subset \mathcal{N}_{\lambda}$, it follows using \eqref{lower bound nehari} that
\begin{align*}
     ||{w}_n||_{L^{4}(\R^3)}^{4}=||{w}_n||_{H^1(\R^3)}^2+\lambda^2\int_{\R^3}\rho \phi_{{w}_n}{w}_n^2\geq ||{w}_n||_{H^1(\R^3)}^2\geq S_{4}^{4}>0.
\end{align*}
	Thus, by interpolation it holds
	$$S_4^4\leq \int_{\R^3}|{w}_n|^{4}\leq\left(\int_{\R^3}|{w}_n|^3\right)^{\frac{2}{3}}\left(\int_{\R^3}|{w}_n|^6\right)^{\frac{1}{3}},$$
	and so, by the Sobolev inequality and \eqref{unif bound tilde w}, it follows that we can pick $\alpha > 0$ independent of $\lambda$ such that
	\begin{equation*}
\liminf_{n \rightarrow \infty}\|{w}_n\|^3_{L^3(\R^3)} \geq \alpha.
	\end{equation*}
Moreover, by Lemma \ref{intoutsideballlemma}, we can set $\lambda_*= \max\{\lambda_0,\lambda_1\}$ and  $R_\alpha>0$ such that such that for every $\lambda\geq \lambda_*$ and every $R>R_\alpha$ we have $$\limsup_{n\rightarrow\infty} \norm{{w}_n}^{3}_{L^{3} (\R^3 \setminus B_R)} < \frac{\alpha}{2}.$$ 
Now, since $({w}_n)_{n\in\N}$ is bounded, passing if necessary to a subsequence, we can assume that ${w}_n\rightharpoonup w$ in $E(\R^3)$ and ${w}_n\rightarrow w$ in $L^{3}_{\textrm{loc}}(\mathbb R^3)$. It follows that for every $\lambda\geq \lambda_*$ and $R>R_\alpha$,
    \begin{equation*}
     \norm{w}^{3}_{L^{3} (B_R)} \geq \liminf\limits_{n\rightarrow \infty}\norm{{w}_n}^{3}_{L^{3} (\mathbb R^3)}  -\limsup\limits_{n\rightarrow \infty}\norm{{w}_n}^{3}_{L^{3} (\R^3 \setminus B_R)}   > \frac{\alpha}{2},
    \end{equation*}
    and so $w\not\equiv0$. We now notice that by \eqref{Ekeland 2}, \eqref{Ekeland 3}, and \eqref{unif bound tilde w}, it holds, up to a constant independent of $\lambda$, that
\begin{equation*}
        \begin{split} 
        o(1)&= ||I_{\lambda}'({w}_n)-\xi_nG_{\lambda}'({w}_n)||_{(E(\R^3))'}\\
        &\gtrsim |I_{\lambda}'({w}_n)({w}_n)-\xi_nG_{\lambda}'({w}_n)({w}_n)|\\
        &=|\xi_nG_{\lambda}'({w}_n)({w}_n)|,
        \end{split}
    \end{equation*}
    for some $\xi_n\in\R$. Since $({w}_n)\subset\mathcal{N}_{\lambda}$, by \eqref{G' less that zero}, we have that $G_{\lambda}'({w}_n)({w}_n)<-2S^4_4<0$, and so the above yields $\xi_n\to0$. Moreover, using \eqref{unif bound tilde w} and the inequality
    \[|D(f,g)|^2\leq D(f,f)D(g,g),\]
    where
    \[D(f,g)\coloneqq \int_{\R^3}\int_{\R^3}\frac{f(x)g(y)}{|x-y|}\dif x \dif y,\]
    for $f,g$ measurable and nonnegative functions (see \cite[p.250]{Lieb and Loss}), it follows that $G_{\lambda}'({w}_n)$ is bounded. Taken together, we have that $\xi_nG_{\lambda}'({w}_n)\to0$, and using this and \eqref{Ekeland 3}, we obtain $I_{\lambda}'({w}_n)\to0$. Hence, $({w}_n)_{n\in\N}$ is a Palais-Smale sequence for $I_{\lambda}$ at level $c_{\lambda}^*$, and so, since we have also shown that ${w}_n\rightharpoonup w\not\equiv 0$ in $E(\R^3)$, a standard argument yields that $w$ is a nontrivial critical point of $I_{\lambda}$. Namely, $w\in\mathcal{N}_{\lambda}$, and thus
    \begin{equation}\label{groundstate ineq 1}
        c_{\lambda}^*\leq I_{\lambda}(w).
    \end{equation}
    On the other hand, arguing as in Proposition \ref{lambdaone}, replacing $u_n$, $u$, and $c_{\lambda}$ with ${w}_n$, $w$, and $c_{\lambda}^*$, respectively, for every $\lambda \geq \lambda_*$, we obtain
    \begin{equation}\label{groundstate ineq 2}
         I_{\lambda}(w)\leq c_{\lambda}^*.
    \end{equation}
    For the reader convenience we recall that $\lambda_1$ is chosen in Proposition \ref{lambdaone} so that for every $\lambda\geq\lambda_1$, it holds that $\frac{1}{4}-S_{\lambda}S^{-3}\bar{C}^4\geq0$, where $\bar{C}$ is defined via Lemma \ref{boundedPSforvanishingrho} by the property $||u_n||_{E(\R^3)}<\bar{C}.$ Going through the same argument with $({w}_n)_{n\in\N}$, since $({w}_n)_{n\in\N}$ is bounded by precisely the same uniform constant, namely $||{w}_n||_{E(\R^3)}<\bar{C}$, we conclude that \eqref{groundstate ineq 2} holds for every $\lambda\geq\lambda_*$, as $\lambda_*\geq\lambda_1$ by construction.
    Putting \eqref{groundstate ineq 1} and \eqref{groundstate ineq 2} together yields $$I_{\lambda}(w)=\inf\limits_{v \in \mathcal{N}_{\lambda}} I_{\lambda} (v).$$ Since $I_{\lambda}(w)=I_{\lambda}(|w|)$ and $w\in \mathcal N_{\lambda}$ if and only if $|w|\in \mathcal N_{\lambda}$, we can assume $w\geq0$, and it follows that $w>0$ by Proposition \ref{reg}. This completes the proof. 
\end{proof}

\noindent As a byproduct of the above proof, we have the following
\begin{proposition}[{\bf Constrained Palais-Smale condition under \ref{vanishing_rho}}]\label{CPSconditionVan}
Let $N=3=q$ and $\rho\geq 0$ be locally bounded such that \ref{vanishing_rho} is satisfied and such that $\rho(x)\rightarrow \rho_\infty>\overline{M}$ as $|x|\rightarrow \infty$. Let $\lambda_0$ and $\lambda_1$
be as in Proposition \ref{lambdaone}. Then, for all $\lambda\geq\max\left(\lambda_0,\lambda_1\right),$ the restriction $I_\lambda|_{\mathcal{N}_{\lambda}}$ satisfies the Palais-Smale condition at the level 
 $$c_{\lambda}^*=\inf\limits_{v \in \mathcal{N}_{\lambda}} I_{\lambda} (v).$$ That is, 
 every sequence $(u_n)_{n\in\mathbb N}\subset E(\R^3)\simeq H^{1}(\R^3)$ such that $$I(u_n)\rightarrow c_{\lambda}^*,\qquad \nabla I_\lambda(u_n)|_{\mathcal{N}_{\lambda}}\rightarrow 0\,\,\textrm{in}\,\,H^{-1}(\R^3)$$
is relatively compact.
\end{proposition}
\begin{proof}
The proof follows reasoning exactly as in Remark \ref{PSremark}. We leave out the details.
\end{proof}
 
\section{The case of coercive $\rho$}\label{coercive rho section}
\noindent In the present section $\lambda > 0$ is an arbitrary fixed value, and on $\rho$ we make the assumption that
\begin{enumerate}[label=$\mathbf{(\rho_{\arabic*})}$]
	\setcounter{enumi}{1}
	\item \label{coercive_rho} For every $M > 0$,
	\begin{equation*}
		\abs{x \in \R^N : \rho(x) \leq M} < \infty.
	\end{equation*}
\end{enumerate}

\begin{lemma}[\bf{Compactness property}]\label{compactembeddingofeintolpplus}
	Let $N = 3,4,5$, $\rho \in L^\infty_{\textrm{loc}}(\R^N)$ be nonnegative, satisfying \ref{coercive_rho}, and $q \in (1, \tstar -1)$. Then, $E(\R^N)$ is compactly embedded into $L^{q+1} (\R^N)$.
\end{lemma}

\begin{proof}
	By Lemma \ref{weightedL3boundlemma}, multiplying by $\lambda$ we obtain
	\begin{equation}\label{weightedL3}
		\lambda \intrn \rho(x) \abs{u}^3 \leq \Big( \frac{1}{\omega} \Big)^{\half} \norme{u}{\R^N}^3.
	\end{equation}
    Set
	\begin{equation*}
		A (R) = \{ x \in \R^N : \abs{x} > R, \ \rho (x) \geq M\},
	\end{equation*}
	\begin{equation*}
		B (R) = \{ x \in \R^N : \abs{x} > R, \ \rho (x) < M\}.
	\end{equation*}
	Without loss of generality, assume that $(u_n)_{n\in\N}\subset E(\R^N)$ is such that $\un \weakly 0$. For convenience, write
	\[
		\int_{\R^N \setminus B_R} \abs{\un}^3 = \int_{A (R)} \abs{\un}^3 + \int_{B (R)} \abs{\un}^3 
	\]
	where $B_R$ is a ball of radius $R$ centred at the origin. Fix $\delta > 0$ and pick $M$, $r$, $C$, such that $M > \frac{2}{\lambda \delta}\left(\frac{1}{\omega}\right)^{\frac{1}{2}}\sup_n \norme{\un}{\R^N}^3$, $r=\frac{\tstar}{3}>1$ and
	\begin{equation*}
		C \geq \sup\limits_{u \in E(\R^N) \setminus \{0\}} \frac{\norm{u}_{L^{2^*}(\R^N)}^3}{\norme{u}{\R^N}^3}.
	\end{equation*}
	Let $\frac{1}{r} + \frac{1}{r^{\prime}} = 1$. By Lemma \ref{measureofBgoingtozerolemma}, for every $M > 0$,
	and every $R>0$ large enough, it holds that 
	\begin{equation}
		|B (R)| \leq \Big[ \frac{\delta}{2C \sup_n \norme{\un}{\R^N}^3} \Big]^{r^{\prime}}.
	\end{equation}
	Since $N=3,4,5$, we can pick $r = \frac{\tstar}{3} > 1$ such that by H\"{o}lder inequality it holds that
	\begin{align*}
		\int_{B (R)} \abs{\un}^3 &\leq \Big( \int_{B (R)} \abs{\un}^{2^*} \Big)^{\frac{1}{r}} \Big( \int_{B (R)} 1 \Big)^{\frac{1}{r^{\prime}}}\\
		&\leq \norm{\un}_{L^{2^*} (\R^N)}^3 \cdot | B (R) |^{\frac{1}{r^{\prime}}}\\
		&\leq C \norme{\un}{\R^N}^3 \cdot | B (R) |^{\frac{1}{r^{\prime}}} \leq \frac{\delta}{2},
	\end{align*}
	Moreover, by our choice of $M$ and \eqref{weightedL3}, we see that
	\begin{equation*}
		\int_{A (R)} \abs{\un}^3  \leq \frac{1}{\lambda M}\left(\frac{1}{\omega}\right)^{\frac{1}{2}}||u_n||_{E(\R^N)}^3 \leq \frac{\delta}{2}.
	\end{equation*}
By the classical Rellich theorem, and since $\delta$ was arbitrary, this is enough to prove our lemma for $q = 2$. By interpolation the case $q \neq 2$ follows immediately, and this concludes the proof.
\end{proof}

Using the above lemma, and for $q\geq 3,$ it is easy to see that the Palais-Smale condition holds for $I_{\lambda}$ at any level. 

\begin{lemma}[\bf{Palais-Smale condition}]\label{pscconditionfori}
	Let $N = 3$, $\rho \in L^\infty_{\textrm{loc}}(\R^N)$ be nonnegative, satisfying \ref{coercive_rho}, and $q \in [3, \tstar -1)$. Then, $I_{\lambda}$ satisfies the Palais-Smale condition at every level $c\in \R$.
\end{lemma}
\begin{proof}
Since by Lemma \ref{compactembeddingofeintolpplus} the embedding of $E(\R^3)$ into $L^{q+1}(\R^3)$ is compact, using Lemma \ref{nonlocalBL}, the conclusion follows arguing as in \cite[p. 1077]{Bonheure and Mercuri}.
\end{proof}

\subsection{Proof of Theorem \ref{existenceandleastenergy} and Theorem \ref{mountainpasssoltheoremlowp}}\label{proofs of existence theorems section}

\begin{proof}[Proof of Theorem \ref{existenceandleastenergy}]
    Using Lemma \ref{mpgfori} and Lemma \ref{pscconditionfori}, the Mountain-Pass Theorem yields the existence a mountain-pass type solution for all $q\in[3,\tstar-1)$. Namely, there exists $u \in E(\R^N)$ such that $I_{\lambda}(u) = c_{\lambda}$ and $I_{\lambda}'(u) = 0$, where $c_{\lambda}$ is given in \eqref{mountainpasslevelforI}. For $q > 3$, the mountain-pass level $c_{\lambda}$ has the characterisation
    \begin{equation*}
        c_{\lambda} = \inf_{u \in \mathcal{N}_{\lambda}} I_{\lambda}(u), \quad \mathcal{N}_{\lambda} := \{ u \in E(\R^N) \setminus \{0\} \ | \ I_{\lambda}' (u)(u) = 0 \},
    \end{equation*}
    and it follows that $u$ is a groundstate solution of $I_{\lambda}$. Since $I_{\lambda}(u)=I_{\lambda}(|u|)$, we can assume $u\geq0$, and so $u>0$ by the strong maximum principle, Proposition \ref{reg}. For $q=3$, we can show the existence of a positive mountain-pass solution applying the general min-max principle \cite[p.41]{Willem book}, and observing that, in our context, we can restrict to admissible curves $\gamma$'s which map into the positive cone $P\coloneqq \{u\in E(\R^3) : u\geq0\}.$ In fact, arguing as in \cite[p.481]{Mercuri and Willem}, since $I_{\lambda}$ satisfies the mountain-pass geometry by Lemma \ref{mpgfori}, it is possible to select a Palais-Smale sequence $(u_n)_{n\in\N}$ at the level $c_{\lambda}$ such that $$\textrm{dist}(u_n,P)\rightarrow 0,$$ from which it follows that $(u_n)_-\rightarrow 0$ in $L^{6}(\R^3),$ see also \cite[Lemma 2.2]{Bonheure Di Cosmo Mercuri}. Then, by construction and up to a subsequence, there exists a weak limit $u\geq 0,$ and hence, by Lemma \ref{pscconditionfori} a nontrivial nonnegative solution, the positivity of which holds by Proposition \ref{reg}. \newline
    The existence of a positive groundstate can be shown with a mild modification to the proof of Theorem \ref{groundstate sol}, using here that all the relevant convergence statements hold for any fixed $\lambda>0$ as a consequence of assumption \ref{coercive_rho} and Lemma \ref{compactembeddingofeintolpplus}.
  This is enough to conclude.
\end{proof}

\begin{proof}[Proof of Theorem \ref{mountainpasssoltheoremlowp}]
We can argue as in \cite[Theorem $1.3$]{Mercuri and Tyler}, based on \cite{Jeanjean and Tanaka} and on the compactness of the embedding of $E(\R^N)$ into $L^{q+1}(\R^N).$ The latter is provided in our context by Lemma \ref{compactembeddingofeintolpplus}. By these, there exists an increasing sequence $\mu_n \rightarrow 1$ and $(\un)_{n\in\N} \in E(\R^N)$ such that $I_{\mu_n, \lambda} (\un) = c_{\mu_n, \lambda}$ and $I_{\mu_n, \lambda}' (\un) = 0,$ where $I_{\mu_n, \lambda}$ and $c_{\mu_n, \lambda}$ are defined as in \eqref{definition I mu lambda} and \eqref{mp level low q}. By Lemma \ref{pohozaevlemma}, we see that
	\begin{equation}\label{poh soe ineq}
	    \frac{N-2}{2}\intrn (\abs{\nabla \un}^2 + \un^2) + \left( \frac{N+2+2k}{4} \right) \intrn \rho(x) \phi_{\un} \un^2 - \frac{N\mu_n}{q+1} \intrn \abs{\un}^{q+1} \leq 0.
	\end{equation}
	Setting $\alpha_n = \intrn (\abs{\nabla \un}^2 + \un^2)$, $\gamma_n = \lambda^2 \intrn \rho(x) \phi_{\un} \un^2$, $\delta_n = \mu_n \intrn \abs{\un}^{q+1},$ we can put together the equalities $I_{\mu_n, \lambda} (\un) = c_{\mu_n, \lambda}$ and  $I_{\mu_n, \lambda}' (\un)(\un) = 0$ with \eqref{poh soe ineq} obtaining the system
	\begin{equation}\label{SOE existence low q}
	    \begin{cases}
            \begin{array}{c c c c c c c}
	            \alpha_n & + & \gamma_n & - & \delta_n & = & 0,\\
	            \half \alpha_n & + & \frac{1}{4} \gamma_n & - &
	            \overqplus \delta_n & = & c_{\mu_n, \lambda},\\
	            \frac{N-2}{2} \alpha_n & + & \left( \frac{N+2+2k}{4} \right) \gamma_n & - & \frac{N}{q+1} \delta_n & \leq & 0,
	        \end{array}
	   \end{cases}
	\end{equation}
	
	\noindent which yields
	\begin{equation*}
	    \delta_n \leq \frac{c_{\mu_n, \lambda}(6-N + 2k) (q+1)}{2(q-2) + k(q-1)}, \quad \gamma_n \leq \frac{2c_{\mu_n, \lambda}\big( 2(q+1)-N (q-1)\big)}{2(q-2) + k(q-1)}, \quad \text{and} \ \alpha_n = \delta_n - \gamma_n.
	\end{equation*}
	We note that $k > \frac{-2(q-2)}{(q-1)} > \frac{N-6}{2}$ since $q<\tstar-1$, and so since $\alpha_n, \gamma_n, \delta_n$ are all nonnegative, it follows that $\alpha_n, \gamma_n, \delta_n$ are all bounded. Hence the sequence $(\un)_{n \in \N}$ is bounded and there exists $u \in E(\R^N)$ such that, up to a subsequence, $\un \weakly u$ in $E(\R^N)$. Using Lemma \ref{compactembeddingofeintolpplus} and arguing as in \cite[Theorem 1]{Bonheure and Mercuri} we obtain that $\norme{\un}{\R^N}^2 \rightarrow \norme{u}{\R^N}^2$ and
	\begin{equation}\label{convergenceofmufunctional}
	    c_{\mu_n, \lambda} = I_{\mu_n, \lambda} (\un) \rightarrow I_{\lambda} (u).
	\end{equation}
	It follows that $\un \rightarrow u$ in $E(\R^N)$,
which combined with the left-continuity property of the levels \cite[Lemma $2.2$]{Ambrosetti and Ruiz}, namely  $c_{\mu_n, \lambda} \rightarrow c_{1, \lambda}={c}_{\lambda}$ as $\mu_n \nearrow 1,$ yields $I_{\lambda} (u) = {c}_{\lambda}.$ Since $u$ is a critical point by the weak convergence, it follows that $u$ is mountain-pass solution. Finally, the existence of a groundstate solution is based on minimising over the set of nontrivial critical points of $I_{\lambda},$ and carrying out an identical argument to the above to show the strong convergence of such a minimising sequence, again using Lemma \ref{compactembeddingofeintolpplus}. This concludes the proof.
\end{proof}
 
\subsection{Proof of Theorem \ref{least energy pure homog rho}}
Under an additional hypotheses on $\rho$, we now prove that the energy level of the groundstate solutions coincide with the mountain-pass level.
\begin{proof}[Proof of Theorem \ref{least energy pure homog rho}]
By Proposition \ref{variational characterisation MP level low q}, it holds that
$${c}_{\lambda}=\inf_{u\in \bar{\mathcal{M}}_{\lambda,\nu}}I_{\lambda}(u),$$
where $\bar{\mathcal{M}}_{\lambda,\nu}$ is defined in \eqref{definition of M lambda}. Since $J_{\lambda,\nu}(u)=0$ is equivalent to the Pohozaev equation given by Lemma \ref{pohozaevlemma} minus the equation $\nu I_{\lambda}'(u)(u)=0$, it is clear that $\bar{\mathcal{M}}_{\lambda,\nu}$ contains all of the critical points of $I_{\lambda}$, and thus the mountain-pass solutions that we find in Theorem \ref{existenceandleastenergy} ($q=3$) and Theorem \ref{mountainpasssoltheoremlowp} ($q<3$) are groundstates. This completes the proof.
\end{proof}
\color{black}
\section{Multiplicity results: coercive $\rho$}\label{5.1}
In the current section, we discuss the existence of high energy solutions in the case $\rho$ satisfies \ref{coercive_rho}. Throughout what follows, we denote the unit ball in $E(\R^N)$ by $B_1$. Moreover, since $\lambda$ does not play any role and can be fixed arbitrarily under assumption \ref{coercive_rho}, we set $\lambda\equiv1$ for the sake of simplicity and define
\begin{equation*}
	I(u) \coloneqq \frac{1}{2}\int_{\R^N}(|\nabla u|^2 + u^2)+\frac{1}{4}\int_{\R^N}\int_{\R^N}\frac{u^2(x)\rho (x) u^2(y) \rho (y)}{|x-y|^{N-2}}\dif x\dif y -\frac{1}{q+1}\int_{\R^N}|u|^{q+1}.
\end{equation*}

\subsection{Preliminaries}\label{prelims for multiplicity section}
We will now discuss some preliminaries that will be used in proving both Theorem \ref{first multiplicity theorem} and \ref{partial multiplicity result low q}. Following \cite{Ambrosetti and Rabinowitz} we set
\begin{equation}\label{A hat}
\hat{A}_0 =\{u\in E(\R^N) : I(u)\geq0\},
\end{equation}
and
\begin{align}\label{gamma star}
\Gamma^* =\{&h\in C(E(\R^N),E(\R^N)): h(0)=0, h \text{ is an odd homeomorphism of } E(\R^N) \\
&\text{ onto } E(\R^N), h(B_1)\subset \hat{A}_0\}.\nonumber
\end{align}
In the next lemma, we establish a result that allows us to obtain high energy solutions in our Banach space setting. Before stating the lemma, we note that since the biorthogonal system given by Lemma \ref{separability} is fundamental, then, for any $m\in\N$, it holds that 
\[E(\R^N)=\text{span}\{e_1,\dots, e_m\}\oplus \overline{\text{span}}\{e_{m+1}, \dots\}.\] 
Thus, throughout what follows we set
\[E_m= \text{span}\{e_1,\dots, e_m\},\]
\[E_m^{\perp}=\overline{\text{span}}\{e_{m+1}, \dots\},\]
and note that, for any $m\in\N$, $E_m$ and $E_m^{\perp}$ define algebraically and topologically complementary subspaces of $E(\R^N)$.

\begin{lemma}[\bf{Divergence of min-max levels $d_m$}]\label{divergence of critical levels}
Let $N \geq 3$ and $q>1$. Suppose $\rho\in L^{\infty}_{loc}(\R^N)$ is nonnegative, satisfying \ref{coercive_rho}. Define 
\begin{equation}\label{def of dm}
d_m\coloneqq \sup_{h\in \Gamma^*}\inf_{u\in\partial B_1 \cap E_{m-1}^{\perp}} I(h(u)),
\end{equation}
where $\Gamma^*$ is given by \eqref{gamma star}. Then, $d_m\to+\infty$ as $m\to+\infty$.
\end{lemma}
\begin{proof}
First we set 
\[T= \left\{u\in E(\R^N)\setminus \{0\}:||u||_{H^1(\R^N)}^2=||u||_{L^{q+1}(\R^N)}^{q+1}\right\}\]
and
\[\tilde{d}_m= \inf_{u\in T\cap E_m^{\perp}}||u||_{E(\R^N)},\]
and claim that $\tilde{d}_m\to+\infty$ as $m\to+\infty$. To see this, assume to the contrary that there exists $u_m\in T \cap E_m^{\perp}$ and some $d>0$ such that $||u_m||_{E(\R^N)}\leq d$ for all $m\in\N$. Since $<e_n^*,u_m>=0$ for all $m\geq n$ and the $e_n^*$'s are total by Lemma \ref{separability}, then it follows that $u_m\rightharpoonup 0$ in $E(\R^N)$ (see e.g.\ \cite{Szulkin}). Since $E(\mathbb R^N)$ is compactly embedded into $L^{q+1}(\mathbb R^N)$ by Lemma \ref{compactembeddingofeintolpplus}, it follows that $u_m\to0$ in $L^{q+1}(\mathbb R^N)$. However, since $u_m\in T$, it follows from the Sobolev inequality that 
\[||u_m||_{H^1(\mathbb R^N)}^{q+1}\geq S_{q+1}^{q+1}||u_m||_{L^{q+1}(\mathbb R^N)}^{q+1}=S_{q+1}^{q+1}||u_m||_{H^{1}(\mathbb R^N)}^{2},\]
from which we deduce
\[||u_m||_{L^{q+1}(\mathbb R^N)}^{q+1}\geq S_{q+1}^{2(q+1)/(q-1)}>0.\]
This shows that $u_m$ is bounded away from $0$ in $L^{q+1}(\mathbb R^N)$, a contradiction, and so we have proved that
\begin{equation}\label{divergence tilde d}
\tilde{d}_m\to+\infty \text{ as } m\to+\infty.
\end{equation}
Now notice that since $E_m$ and $E_m^{\perp}$ are complementary subspaces, it holds that there exists a $\bar{C}\geq1$ such that each $u\in B_1$ can be uniquely written as 
\begin{equation}\label{decomposition of E functions}
u=v+w, \text{ with } v\in E_m, w\in E_m^{\perp},
\end{equation}
\begin{equation}\label{v less than C}
||v||_{E(\mathbb R^N)}\leq \bar{C}||u||_{E(\mathbb R^N)}\leq\bar{C},
\end{equation}
\begin{equation}\label{w less than C}
||w||_{E(\mathbb R^N)}\leq \bar{C}||u||_{E(\mathbb R^N)}\leq\bar{C},
\end{equation}
as a consequence of the open mapping theorem, see \cite[p.37]{Brezis Book}. Define $h_m:E_m^{\perp}\to E_m^{\perp}$ by
\[h_m(u)=(\bar{C}K)^{-1}\tilde{d}_m u,\]
where 
\[K>\max\left\{1,\left(\frac{4}{q+1}\right)^{\frac{1}{q-1}}\right\},\]
and note that $h_m$ is an odd homeomorphism of $E_m^{\perp}$ onto $E_m^{\perp}$. Now, for any $u\in E(\R^N) \setminus \{0\}$, there exists a unique $\beta(u)>0$ such that $\beta(u)u\in T$, namely
\begin{equation}\label{def of beta u}
\beta(u) = \left(\frac{||u||_{H^1(\mathbb R^N)}^2}{||u||_{L^{q+1}(\mathbb R^N)}^{q+1}}\right)^{\frac{1}{q-1}}.
\end{equation}
If we define
\[I_0(u)= \frac{1}{2} ||u||_{H^1(\mathbb R^N)}^2-\frac{1}{q+1}||u||_{L^{q+1}(\mathbb R^N)}^{q+1},\]
then for each $u\in E(\R^N) \setminus \{0\}$, it holds that
\[I_0(tu)=\frac{t^2}{2} ||u||_{H^1(\mathbb R^N)}^2-\frac{t^{q+1}}{q+1}||u||_{L^{q+1}(\mathbb R^N)}^{q+1}\]
is a monotone increasing function for $t\in[0,\beta(u)]$ with a maximum at $t=\beta(u)$. Note that for each $u\in (E_m^{\perp}\cap B_{\bar{C}})\setminus \{0\}$, by the definition of $\bar{d}_m$ and $\beta(u)$, we have
\begin{align}\label{d m over beta u}
    \bar{C}^{-1}\tilde{d}_m\leq \bar{C}^{-1} ||\beta(u)u||_{E(\mathbb R^N)}\leq \beta(u),
\end{align}
and so since $K\geq1$, it holds that
\[(\bar{C}K)^{-1}\tilde{d}_m\leq \bar{C}^{-1}\tilde{d}_m\leq \beta(u),\quad \text{ for all } u\in (E_m^{\perp}\cap B_{\bar{C}})\setminus \{0\}.\]
Putting everything together, it follows that
\[I_0(h_m(u))=I_0((\bar{C}K)^{-1}\tilde{d}_mu)>0\quad \text{ for all } u\in (E_m^{\perp}\cap B_{\bar{C}})\setminus \{0\}.\]
Moreover,
\[h_m(0)=0.\]
Therefore,
\begin{equation}\label{image of Em orthog intersect ball}
h_m(E_m^{\perp}\cap B_{\bar{C}})\subset \left\{u\in E(\R^N): I_0(u)\geq 0\right\}.
\end{equation}
Now, for each $m\in\N$ and some $\delta>0$, define $\tilde{h}_m:E_m\times E_m^{\perp}\to E_m\times E_m^{\perp}$ by
\[\tilde{h}_m([v,w])=[\delta v,\, (\bar{C}K)^{-1}\tilde{d}_mw ].\]
Notice that $\tilde{h}_m$ is an odd homeomorphism of $E_m\times E_m^{\perp}$ onto $E_m\times E_m^{\perp}$. Moreover, by \eqref{decomposition of E functions}, the function $g_m:E_m\times E_m^{\perp}\to E(\R^N)$ defined by
\[g_m([v,w])=v+w,\]
is an odd homeomorphism. Hence, defining $H_m:E(\R^N)\to E(\R^N)$ as
\[H_m= g_m\circ \tilde{h}_m \circ g^{-1}_m,\]
we see that $H_m$ is an odd homeomorphism of $E(\R^N)$ onto $E(\R^N)$. By \eqref{decomposition of E functions}-\eqref{w less than C}, it holds that
\[B_1\subseteq g_m(\left\{E_m\cap B_{\bar{C}}\right\}\times\{E_m^{\perp}\cap B_{\bar{C}}\}),\]
and so
\begin{align}\label{image of ball contained in Z}
    H_m(B_1)&\subseteq H_m(g_m(\{E_m\cap B_{\bar{C}}\}\times\{E_m^{\perp}\cap B_{\bar{C}}\}))\\
    &=g_m(\tilde{h}_m(\{E_m\cap B_{\bar{C}}\}\times\{E_m^{\perp}\cap B_{\bar{C}}\}))\nonumber\\
    &=g_m (\{\delta(E_m\cap B_{\bar{C}})\}\times\{\bar{C}^{-1}K^{-1}\tilde{d}_m (E_m^{\perp}\cap B_{\bar{C}})\})\nonumber\\
    &=\left\{u\in E(\R^N): u=v+w, v\in \delta(E_m\cap B_{\bar{C}}), w\in \bar{C}^{-1}K^{-1}\tilde{d}_m (E_m^{\perp}\cap B_{\bar{C}}) \right\}\nonumber\\
    &=:Z_{m,\delta}.\nonumber
\end{align}
Now, fix $m\in\N$. We claim that $$Z_{m,\delta}\subset \left\{u\in E(\R^N): I_0(u) > 0\right\} \cup \{0\}$$ for some $\delta=\delta(m)>0$. To see this, assume, by contradiction, that there exists $\delta_j\to 0$ and $u_j\notin \left\{u\in E(\R^N): I_0(u) > 0\right\} \cup \{0\}$ such that $u_j\in Z_{m,\delta_j}$. Then, by definition of $Z_{m,\delta_j}$, it holds that
\[||u_j||_{E(\R^N)}\leq||v_j||_{E(\R^N)}+||w_j||_{E(\R^N)}\leq \delta_j\bar{C}+K^{-1}\tilde{d}_m,\]
which implies $u_j$ is bounded.
Thus, up to a subsequence $u_j\rightharpoonup \bar{u}$ in $E(\R^N)$ and so it follows that $u_j\rightharpoonup \bar{u}$ in $H^1(\R^N)$. Moreover, since $E(\mathbb R^N)$ is compactly embedded into $L^{q+1}(\mathbb R^N)$ by Lemma \ref{compactembeddingofeintolpplus}, it follows that $u_j\to\bar{u}$ in $L^{q+1}(\mathbb R^N),$ with $||\bar{u}||_{L^{q+1}(\R^N)}^{q+1}>0$ by previous arguments. Thus, by the weak lower semicontinuity of the $H^1(\R^N)$ norm and the strong convergence in $L^{q+1}(\mathbb R^N)$, we deduce that 
\[\frac{1}{2}||\bar{u}||_{H^1(\R^N)}^2 \leq \frac{1}{q+1}||\bar{u}||_{L^{q+1}(\R^N)}^{q+1},\]
which implies $\bar{u}\notin \left\{u\in E(\R^N): I_0(u) > 0\right\} \cup \{0\}$. On the other hand, since $\delta_j\to0$, then $v_j\to 0$. It follows from this and \eqref{image of Em orthog intersect ball} that $\bar{u}\in \bar{C}^{-1}K^{-1}\tilde{d}_m (E_m^{\perp}\cap B_{\bar{C}})\subset \left\{u\in E(\R^N): I_0(u) > 0\right\} \cup \{0\}$. Hence, we have reached a contradiction and so the claim holds. Thus, using this and \eqref{image of ball contained in Z}, for each $m\in\N$, we pick $\delta=\delta(m)>0$ so that 
\begin{align*}
    H_m(B_1)\subset \left\{u\in E(\R^N): I_0(u) > 0\right\} \cup \{0\} \subset \left\{u\in E(\R^N): I(u)\geq 0\right\}=\hat{A}_0,
\end{align*} 
namely $H_m\in\Gamma^*$, where $\hat{A}_0$ and $\Gamma^*$ are given by \eqref{A hat} and \eqref{gamma star}, respectively. We can therefore see that
\begin{equation}\label{lower bound d m}
d_{m+1}= \sup_{h\in \Gamma^*}\inf_{u\in\partial B_1 \cap E_{m}^{\perp}} I(h(u))\geq \inf_{u\in\partial B_1 \cap E_{m}^{\perp}} I(H_{m}(u)).
\end{equation}
Now take $u\in \partial B_1 \cap E_{m}^{\perp}$. Then, using \eqref{def of beta u}, \eqref{d m over beta u} and the fact that $\int_{\R^N}\rho \phi_u u^2 =\omega^{-1}(1-||u||_{H^1(\R^N)}^2)^2$, it holds that
\begin{align*}
    I(H_m(u))&=\frac{1}{2}(\bar{C}^{-1}K^{-1}\tilde{d}_m)^2||u||_{H^1(\R^N)}^2+\frac{1}{4}(\bar{C}^{-1}K^{-1}\tilde{d}_m)^4\int_{\R^N}\rho\phi_u u^2\\
    &\qquad\qquad-\frac{1}{q+1}(\bar{C}K)^{-q-1}\tilde{d}_m^{q+1}||u||_{L^{q+1}(\R^N)}^{q+1}\\
    &=\frac{1}{2}(\bar{C}^{-1}K^{-1}\tilde{d}_m)^2||u||_{H^1(\R^N)}^2+\frac{1}{4}(\bar{C}^{-1}K^{-1}\tilde{d}_m)^4\int_{\R^N}\rho\phi_u u^2\\
    &\qquad\qquad-\frac{(\bar{C}K)^{-q-1}\tilde{d}_m^{2}}{q+1}\left(\frac{\tilde{d}_m}{\beta(u)}\right)^{q-1}||u||_{H^1(\R^N)}^{2}\\
    &\geq\frac{1}{2}(\bar{C}^{-1}K^{-1}\tilde{d}_m)^2\left(1-\frac{2K^{1-q}}{q+1}\right)||u||_{H^1(\R^N)}^2\\
    &\qquad\qquad+\frac{1}{4\omega}(\bar{C}^{-1}K^{-1}\tilde{d}_m)^4\left(1-||u||_{H^1(\R^N)}^{2}\right)^2\\
    &\geq \min \left\{K_1\tilde{d}_m^2,K_2\tilde{d}_m^4 \right\} \left(||u||_{H^1(\R^N)}^{4}-||u||_{H^1(\R^N)}^{2}+1\right)\\
    &\geq \frac{3}{4}\min \left\{K_1\tilde{d}_m^2,K_2\tilde{d}_m^4 \right\},
\end{align*}
where $K_1\geq \frac{1}{4\bar{C}^{2}K^{2}}$ by our choice of $K$ and $K_2= \frac{1}{4\omega\bar{C}^4K^4}$. Finally, using this, \eqref{lower bound d m}, and \eqref{divergence tilde d}, we obtain
\begin{align*}
    d_{m+1}&\geq \inf_{u\in\partial B_1 \cap E_{m}^{\perp}} I(H_{m}(u))\\
    &\geq \frac{3}{4}\min \left\{K_1\tilde{d}_m^2,K_2\tilde{d}_m^4 \right\}\to+\infty,\quad \text{ as } m\to+\infty.
\end{align*}
This completes the proof.
\end{proof}

\subsection{Proof of Theorem \ref{first multiplicity theorem}}\label{proof of multiplicity theorem section}
In order to prove Theorem \ref{first multiplicity theorem}, we will need some background material including the notion of the Krasnoselskii-genus and its properties. Throughout what follows we let $G$ be a compact topological group. Following \cite{Costa}, we begin with a number of definitions that we will need before introducing the notion of the Krasnoselskii-genus.
\begin{definition}[\bf{Isometric representation}]
The set $\{T(g) : g\in G\}$ is an {isometric representation} of $G$ on $E$ if $T(g):E\to E$ is an isometry for each $g\in G$ and the following hold:
\begin{enumerate}[label=(\roman*)]
\item $T(g_1+g_2)=T(g_1)\circ T(g_2)$ for all $g_1,g_2\in G$
\item $T(0) =I$, where $I:E\to E$ is the identity map on $E$
\item $(g,u) \mapsto T(g)(u)$ is continuous.
\end{enumerate}
\end{definition}
\begin{definition}[\bf{Invariant subset}]
A subset $A\subset E$ is {invariant} if $T(g)A=A$ for all $g\in G$.
\end{definition}
\begin{definition}[\bf{Equivariant mapping}]
A mapping $R$ between two invariant subsets $A_1$ and $A_2$, namely $R:A_1\to A_2$, is said to be {equivariant} if $R\circ T(g)=T(g)\circ R$ for all $g\in G$.
\end{definition}
\begin{definition}[\bf{The class $\mathcal{A}$}]
We denote the class of all closed and invariant subsets of $E$ by $\mathcal{A}$. Namely,
\[\mathcal{A} \coloneqq \{A\subset E : A \text{ closed},\, T(g)A=A \,\, \forall g\in G\}.\]
\end{definition}
\begin{definition}[\bf{${G}$-index with respect to $\mathcal{A}$}]
A {${G}$-index} on $E$ with respect to $\mathcal{A}$ is a mapping $\text{ind}:\mathcal{A}\to\N\cup\{+\infty\}$ such that the following hold:
\begin{enumerate}[label=(\roman*)]
\item $\text{ind}(A)=0$ if and only if $A=\emptyset$.
\item If $R:A_1\to A_2$ is continuous and equivariant, then $\text{ind}(A_1)\leq \text{ind}(A_2)$.
\item $\text{ind}(A_1\cup A_2)\leq \text{ind}(A_1)+ \text{ind}(A_2)$.
\item If $A\in \mathcal{A}$ is compact, then there exists a neighbourhood $N$ of $A$ such that $N\in\mathcal{A}$ and $\text{ind}(N)=\text{ind}(A)$.
\end{enumerate}
\end{definition}
With these definitions in place, we are ready to introduce the concept of the Krasnoselskii-genus.
\begin{lemma}[\bf{The Krasnoselskii-genus}]
Let $G=\Z_2=\{0,1\}$ and define $T(0)=I$, $T(1)=-I$, where $I:E\to E$ is the identity map on $E$. Given any closed and symmetric with respect to the origin subset $A\in\mathcal{A}$, define $\gamma(A)=k\in\N$ if $k$ is the smallest integer such that there exists some odd mapping $\varphi \in C(A,\R^k \setminus\{0\})$. Moreover, define $\gamma(A)=+\infty$ if no such mapping exists and $\gamma(\emptyset)=0$. Then, the mapping $\gamma:\mathcal{A}\to \N\cup \{+\infty\} $ is a $\Z_2$-index on $E$, called the Krasnoselskii-genus.
\end{lemma}
\begin{proof}
See the proof of Proposition $2.1$ in \cite{Costa}.
\end{proof}
The next lemma gives a property of the Krasnoselskii-genus relevant for us to obtain our multiplicity result.
\begin{lemma}[\bf{Multiplicity from the Krasnoselskii-genus}]\label{infinitely many sols lemma}
Assume $A\in \mathcal{A}$ is such that $0\notin A$ and $\gamma(A)\geq2$. Then, $A$ has infinitely many points.
\end{lemma}
\begin{proof}
See the proof of Proposition $2.2$ in \cite{Costa}.
\end{proof}

For the proof of Theorem \ref{first multiplicity theorem}, we recall a classical result of Ambrosetti and Rabinowitz, \cite{Ambrosetti and Rabinowitz}.
\begin{theorem}[\cite{Ambrosetti and Rabinowitz}; \,\bf{Min-max setting high $q$}]\label{Ambro Rab Mult}
Let $I\in C^1(E(\R^N),\R^N)$ satisfy the following:
\begin{enumerate}[label=(\roman*)]
\item $I(0)=0$ and there exists constants $R,a>0$ such that $I(u)\geq a$ if $||u||_{E(\R^N)}=R$
\item If $(u_n)_{n\in\N}\subset E(\R^N)$ is such that $0<I(u_n)$, $I(u_n)$ bounded above, and $I'(u_n)\to 0$, then $(u_n)_{n\in\N}$ possesses a convergent subsequence
\item $I(u)=I(-u)$ for all $u\in E(\R^N)$
\item For a nested sequence $E_1\subset E_2\subset \cdots$ of finite dimensional subspaces of $E(\R^N)$ of increasing dimension, it holds that $E_i \cap \hat{A}_0$ is bounded for each $i=1,2,\ldots$, where $\hat{A}_0$ is given by \eqref{A hat}
\end{enumerate}
Define
\[b_m=\inf_{K\in\Gamma_m}\max_{u\in K} I(u),\]
with
\begin{align*}
\Gamma_m=&\{K\subset E(\R^N) : K \text{ is compact and symmetric with respect to the origin and for}\\
&\text{ all } h\in\Gamma^*, \text{ it holds that } \gamma(K\cap h(\partial B_1))\geq m\},
\end{align*} 
where $\Gamma^*$ is given by \eqref{gamma star}.
Then, for each $m\in\N$, it holds that $0<a \leq b_m \leq b_{m+1}$ and $b_m$ is a critical value of $I$. Moreover, if $b_{m+1}=\cdots =b_{m+r}=b$, then $\gamma(K_b)\geq r$,
where 
\[K_b \coloneqq\{u\in E(\R^N): I(u)=b, \, I'(u)=0\},\]
is the set of critical points at any level $b>0$.
\end{theorem}
\begin{proof}
See \cite[Theorem $2.8$]{Ambrosetti and Rabinowitz}.
\end{proof}
We are now in position to prove Theorem \ref{first multiplicity theorem}.
\begin{proof} [Proof of Theorem \ref{first multiplicity theorem}]
We aim to apply Theorem \ref{Ambro Rab Mult} and therefore must verify that $I$ satisfies assumptions $(i)$-$(iv)$ of this theorem. By Lemma \ref{mpgfori}, $I$ satisfies the Mountain-Pass Geometry and thus $(i)$ holds. By Lemma \ref{pscconditionfori}, $(ii)$ holds. Clearly, $(iii)$ holds due to the structure of the functional $I$. We now must show that $(iv)$ holds. We first notice by straightforward calculations that for any $u\in \partial B_1$ and any for $t>0$, it holds that
\begin{align*}
I(tu)&=\frac{t^2}{2}||u||^2_{H^1(\R^N)}+\frac{t^4}{4}\int_{\R^N}\rho\phi_u u^2-\frac{t^{q+1}}{q+1}\int_{\R^N}|u|^{q+1}\\
&=\frac{t^2}{2}\left(||u||^2_{H^1(\R^N)}+\frac{t^2}{2}\int_{\R^N}\rho\phi_u u^2-\frac{2t^{q-1}}{q+1}\int_{\R^N}|u|^{q+1}\right).
\end{align*}
We now set
\[\alpha\coloneqq||u||^2_{H^1(\R^N)}>0, \quad \beta \coloneqq \frac{1}{2} \int_{\R^N}\rho\phi_u u^2\geq0, \quad \gamma \coloneqq \frac{2}{q+1}\int_{\R^N}|u|^{q+1}>0,\]
and look for positive solutions of 
\[\frac{t^2}{2}(\alpha+\beta t^2-\gamma t^{q-1})=0.\]
Since $q>3$, it holds that $\alpha+\beta t^2-\gamma t^{q-1}=0$ has a unique solution $t=t_u>0$. That is, we have shown that for each $u\in \partial B_1$, there exists a unique $t=t_u>0$ such that $I$ satisfies
\begin{align*}
&I (t_uu)=0 \\
&I (tu)>0, \,\, \forall t<t_u \\
&I (tu)<0, \,\, \forall t>t_u.
\end{align*}
Now, for any $m\in\N$, we choose $E_m$ a $m$-dimensional subspace of $E(\R^N)$ in such a way that $E_m\subset E_{m'}$ for $m<m'$. Moreover, for any $m\in \N$, we set 
\[W_m\coloneqq \{w\in E(\R^N) : v=tu,\,\, t\geq 0,\,\, u\in \partial B_1 \cap E_m\}.\]
Then, the function $h:E_m \to W_m$ given by
\[h(z)= t \frac{z}{||z||}, \quad \text{with } t=||z||\]
defines a homeomorphism from $E_m$ onto $W_m$, and so $W_1\subset W_2\subset \cdots$ is a nested sequence of finite dimensional subspaces of $E(\R^N)$ of increasing dimension. We also notice that
\[T_m\coloneqq \sup_{u\in \partial B_1 \cap E_m} t_u <+\infty\]\\
since $\partial B_1 \cap E_m$ is compact. So, for all $t>T_m$ and $u\in\partial B_1 \cap E_m$, it holds that $I (tu)<0$, and thus $W_m \cap \hat{A}_0$ is bounded, where $\hat{A}_0$ is given by \eqref{A hat}. Since this holds for arbitrary $m\in\N$, we have shown that $(iv)$ holds. Hence, we have shown that Theorem \ref{Ambro Rab Mult} applies to the functional $I$. If $b_m$ are distinct for $m=1,\dots, j$ with $j\in \N$, we obtain $j$ distinct pairs of critical points corresponding to critical levels $0<b_1<b_2<\cdots<b_j$. If $b_{m+1}=\cdots =b_{m+r}=b$, then $\gamma(K_b)\geq r\geq 2$. Moreover, $0\notin K_b$ since $b>0=I(0)$. Further, $K_b$ is invariant since $I$ is an invariant functional and $K_b$ is closed since $I$ satisfies the Palais-Smale condition, and so $K_b\in \mathcal{A}$. Therefore, by Lemma \ref{infinitely many sols lemma}, $K_b$ possesses infinitely many points. Finally, we note that by \cite[Theorem $2.13$]{Ambrosetti and Rabinowitz}, for each $m\in\N$, it holds that 
\[d_m\leq b_m,\]
where $d_m$ is defined in \eqref{def of dm}. It therefore follows from Lemma \ref{divergence of critical levels} that 
\[b_m\to+\infty, \text{ as } m\to+\infty.\]
This concludes the proof.
\end{proof}

\subsection{Proof of Theorem \ref{partial multiplicity result low q}}\label{proof of multiplicity theorem low q section}

Before proving Theorem \ref{partial multiplicity result low q}, we must establish some preliminary results that we will need to use. The first lemma that we recall will give us an abstract definition of the min-max levels and some properties.

\begin{lemma}[\cite{Ambrosetti and Ruiz};\, \bf{Abstract min-max setting for low $q$}]\label{Ambro Ruiz critical levels}
Consider a Banach space $E$, and a functional $\Phi_{\mu}:E\to \R$ of the form $\Phi_{\mu}(u)=\alpha(u)-\mu \beta(u)$, with $\mu>0$. Suppose that $\alpha$, $\beta \in C^1$ are even functions, $\lim_{||u||\to+\infty} \alpha(u)=+\infty$, $\beta(u)\geq 0$, and $\beta$, $\beta'$ map bounded sets onto bounded sets. Suppose further that there exists $K\subset E$ and a class $\mathcal{F}$ of compact sets in $E$ such that: \\

\noindent ($\mathcal{F}.1$) $K\subset A$ for all $A\in \mathcal{F}$ and $\sup_{u\in K} \Phi_{\mu}(u)<c_{\mu}$, where $c_{\mu}$ is defined as:
\begin{equation}\label{def of c mu}
c_{\mu}\coloneqq \inf_{A\in\mathcal{F}}\max_{u\in A}\Phi_{\mu}(u).
\end{equation}
 ($\mathcal{F}.2$) If $\eta\in C([0,1] \times E, E)$ is an odd homotopy such that 
\begin{itemize}
\item $\eta(0, \cdot)=I$, where $I:E\to E$ is the identity map on $E$ 
\item$\eta(t,\cdot)$ is a homeomorphism 
\item $\eta(t,x)=x$ for all $x\in K$,
\end{itemize}
then $\eta(1,A)\in\mathcal{F}$ for all $A\in\mathcal{F}$.\\

\noindent Then, it holds that the mapping $\mu \mapsto c_{\mu}$ is non-increasing and left-continuous, and therefore is almost everywhere differentiable.
\end{lemma}

\begin{proof}
See \cite[Lemma $2.2$]{Ambrosetti and Ruiz}.
\end{proof}
Under the hypotheses of the previous lemma, we can now define the set of values of $\mu\in\left[\frac{1}{2},1\right]$ such that $c_{\mu}$, given by \eqref{def of c mu}, is differentiable. Namely, we define
\begin{equation*}\label{set of mu}
\begin{split}
\mathcal{J}\coloneqq \bigg\{\mu \in \left[\frac{1}{2},1\right] \, :\, \text{the mapping } \mu \mapsto c_{\mu} \text{ is differentiable}\bigg\}.
\end{split}
\end{equation*}

\begin{corollary}[\bf{On density of perturbation values $\mu$}]\label{density corollary}

The set $\mathcal{J}$ is dense in $\left[\frac{1}{2},1\right]$.
\end{corollary}

\begin{proof}
Fix $x\in\left[\frac{1}{2},1\right]$ and $\delta>0,$ and denote by $\abs{\cdot}$ the Lebesgue measure. Since  $\left[\frac{1}{2},1\right] \setminus \mathcal{J}$ has zero Lebesgue measure by Lemma \ref{Ambro Ruiz critical levels}, we have
\begin{align*}
\left|\mathcal{J} \cap (x-\delta,x+\delta)\right| 
&=\left|\left[\frac{1}{2},1\right] \cap (x-\delta,x+\delta)\right|>0.
\end{align*}
It follows that $\mathcal{J} \cap (x-\delta,x+\delta)$ is nonempty and so we can choose $y\in\mathcal{J} \cap(x-\delta,x+\delta)$. Since $x$ and $\delta$ are arbitrary, this completes the proof.
\end{proof}
With the definition of $\mathcal{J}$ in place, we can also recall another vital result from \cite{Ambrosetti and Ruiz}, which will be used to obtain the boundedness of our Palais-Smale sequences.
\begin{lemma}[\cite{Ambrosetti and Ruiz};\, \bf{Boundedness of Palais-Smale sequences at level $c_{\mu}$}]\label{Ambro Ruiz bounded PS}
For any $\mu\in \mathcal{J}$, there exists a bounded Palais-Smale sequence for $\Phi_{\mu}$ at the level $c_{\mu}$ defined by \eqref{def of c mu}. That is, there exists a bounded sequence $(u_n)_{n\in\N}\subset E(\R^N)$ such that $\Phi_{\mu}(u_n)\to c_{\mu}$ and $\Phi_{\mu}'(u_n)\to 0$.
\end{lemma}
\begin{proof}
See \cite[Proposition $2.3$]{Ambrosetti and Ruiz}.
\end{proof}
Moving toward a less abstract setting, for any $\mu\in\left[\frac{1}{2},1\right]$, we define the perturbed functional $I_{\mu}:E(\R^N)\to\R^N$ as 
\begin{equation}\label{def of I mu}
	I_{\mu}(u) \coloneqq \frac{1}{2}\int_{\R^N}(|\nabla u|^2 + u^2)+\frac{1}{4}\int_{\R^N}\int_{\R^N}\frac{u^2(x)\rho (x) u^2(y) \rho (y)}{|x-y|^{N-2}}\dif x\dif y -\frac{\mu}{q+1}\int_{\R^N}|u|^{q+1}.
\end{equation}
The next result that we will need in order to prove Theorem \ref{partial multiplicity result low q}, follows as a result of Lemma \ref{derivatives of f}.

\begin{lemma}[\bf{On the sign of the energy level of $I_{\mu}$ along certain curves}]\label{use of curves}
Assume $N= 3,4,5$ and $q\in(2,\tstar-1]$. Suppose further that $\rho$ is homogeneous of degree $\bar{k}$, namely, $ \rho(tx)=t^{\bar{k}}\rho(x)$ for all $t>0$, for some 
$$\bar{k}>\max\left\{\frac{N}{4}, \frac{1}{q-1}\right\}\cdot(3-q)-1.$$
Then, there exists $\nu>\max\left\{\frac{N}{2}, \frac{2}{q-1}\right\}$ such that for each fixed $\mu\in\left[\frac{1}{2},1\right]$ and each $u\in E(\R^N)\setminus\{0\}$, there exists a unique $t=t_u>0$ with the property that
\begin{align*}
&I_{\mu} (t_u^{\nu} u({t_u}\cdot))=0, \nonumber \\
&I_{\mu} (t^{\nu} u(t\cdot))>0, \,\, \forall t<t_u, \nonumber \\
&I_{\mu} (t^{\nu} u(t\cdot))<0, \,\, \forall t>t_u,
\end{align*}
where $I_{\mu}$ is defined in \eqref{def of I mu}.
\end{lemma}
\begin{proof}

We first note that under the assumptions on the parameters, we can show that
\[\frac{4\nu-N-2}{2}>\frac{(\nu+1)(3-q)-2}{2}.\]
It follows from this and the lower bound assumption on $\bar{k}$ that we can always find at least one interval
\[\left(\frac{\nu(3-q)-2}{2},\frac{4\nu-N-2}{2}\right), \quad \text{with } \nu>\max\left\{\frac{N}{2},\frac{2}{q-1}\right\},\]
that contains $\bar{k}$. We pick $\nu$ corresponding to such an interval and fix $\mu\in\left[\frac{1}{2},1\right]$. Then, for any $u\in E(\R^N)\setminus\{0\}$ and for any $t>0$,  
using the assumption that $\rho$ is homogeneous of degree $\bar{k}$, we find that
\begin{align*}
I_{\mu} (t^{\nu} u(t\cdot))&=\frac{t^{2\nu+2-N}}{2}\int_{\R^N}|\nabla u|^2 + \frac{t^{2\nu-N}}{2}\int_{\R^N}u^2+\frac{t^{4\nu-N-2}}{4}\int_{\R^N} \int_{\R^N} \frac{u^2(y)\rho(\frac{y}{t})u^2(x)\rho(\frac{x}{t})}{\omega |x-y|^{N-2}}\\
&\qquad-\frac{\mu t^{\nu(q+1)-N}}{q+1}\int_{\R^N}|u|^{q+1}\\
&=\frac{t^{2\nu+2-N}}{2}\int_{\R^N}|\nabla u|^2 + \frac{t^{2\nu-N}}{2}\int_{\R^N}u^2+\frac{t^{4\nu-N-2-2\bar{k}}}{4}\int_{\R^N} \rho\phi_u u^2 -\frac{\mu t^{\nu(q+1)-N}}{q+1}\int_{\R^N}|u|^{q+1}.
\end{align*}
We therefore set
\[a=\frac{1}{2} \int_{\R^N}|\nabla u|^2, \,\, b=\frac{1}{2} \int_{\R^N} u^2,\,\, c=\frac{1}{4} \int_{\R^N}\rho\phi_u u^2, \,\, d=\frac{\mu}{q+1} \int_{\R^N}|u|^{q+1},\]
and consider the polynomial
\begin{equation*}
    f(t)= a{t^{2\nu+2-N}}+bt^{2\nu-N}+c t^{4\nu-N-2-2\bar{k}}-d t^{\nu(q+1)-N}, \quad t\geq0.
    \end{equation*}
Since $u\in E(\R^N)\setminus\{0\}$, we can deduce that $a,b,d>0$ and $c\geq0$, and so, by Lemma \ref{derivatives of f}, it holds that $f$ has a unique critical point corresponding to its maximum. Thus, since $I_{\mu} (t^{\nu} u(t\cdot))=f(t)$ and, by assumptions, $\nu(q+1)-N>2\nu+2-N$ and $\nu(q+1)-N>4\nu-N-2-2\bar{k}$, it follows that there exists a unique $t=t_u>0$ such that the conclusion holds.

\end{proof}

With the previous results established, we are finally in position to prove Theorem \ref{partial multiplicity result low q}.

\begin{proof} [Proof of Theorem \ref{partial multiplicity result low q}]

We first note that by Lemma \ref{use of curves}, we can choose $\nu>\max\left\{\frac{N}{2},\frac{2}{q-1}\right\}$, so that for each $u\in\partial B_1$, there exists a unique $t=t_u>0$ such that $I_{\mu}$ with $\mu=\frac{1}{2}$, defined by \eqref{def of I mu}, satisfies
\begin{align}\label{sign of mathcal I}
&I_{\frac{1}{2}} (t_u^{\nu} u({t_u}\cdot))=0, \nonumber \\
&I_{\frac{1}{2}} (t^{\nu} u(t\cdot))>0, \,\, \forall t<t_u, \nonumber \\
&I_{\frac{1}{2}} (t^{\nu} u(t\cdot))<0, \,\, \forall t>t_u.
\end{align}
Now, for any $m\in\N$, we choose $E_m$ a $m$-dimensional subspace of $E(\R^N)$ in such a way that $E_m\subset E_{m'}$ for $m<m'$. Moreover, for any $m\in \N$, we set 
\[W_m\coloneqq \{w\in E(\R^N) : w= t^{\nu}u(t\cdot),\,\, t\geq 0,\,\, u\in \partial B_1 \cap E_m\}.\]
Then, the function $h:E_m \to W_m$ given by
\[h(e)= t^{\nu}u(t\cdot), \quad \text{with } t=||e||_{E(\R^N)}, \,u=\frac{e}{||e||_{E(\R^N)}},\]
defines an odd homeomorphism from $E_m$ onto $W_m$. We notice that it holds that
\begin{equation}\label{def of Tm}
T_m\coloneqq \sup_{u\in \partial B_1 \cap E_m} t_u <+\infty,
\end{equation}
since $\partial B_1 \cap E_m$ is compact. So, the set
\[A_m= \{w\in E(\R^N) : w=t^{\nu} u(t\cdot),\,\, t\in[0,T_m],\,\, u\in \partial B_1 \cap E_m\}\]
is compact. We now define
\[H\coloneqq\{g:E(\R^N)\to E(\R^N) : g \text{ is an odd homeomorphism and } g(w)=w \, \text{ for all } w\in\partial A_m\},\]
and
\[G_m\coloneqq \{g(A_m): g\in H\}.\]
We aim to verify ($\mathcal{F}.1$) and ($\mathcal{F}.2$) of Lemma \ref{Ambro Ruiz critical levels}. We take $G_m$ as the class $\mathcal{F}$ and $K=\partial A_m$ and define the min-max levels
\[c_{m,\mu}\coloneqq \inf_{A\in G_m}\max_{u\in A} I_{\mu}(u).\]
Then, since $T_m\geq t_u$ for all $u\in \partial B_1 \cap E_m$ by definition, it follows from \eqref{sign of mathcal I} that
\[I_{\mu}(w)\leq I_{\frac{1}{2}} (w)\leq 0, \quad \forall w\in \partial A_m, \,\, \forall \mu\in\left[\frac{1}{2},1\right].\]
Moreover, since $G_m\subset G_{m+1}$ for all $m\in\N$, it holds that $c_{m,\mu}\geq c_{m-1,\mu}\geq\cdots\geq c_{1,\mu}>0$. Taken together, we have shown that
\begin{equation}\label{sup on partial Am less than 0}
\sup_{w\in\partial A_m}I_{\mu}(w)\leq 0 <c_{m,\mu},
\end{equation}
and thus ($\mathcal{F}.1$) is verified. Moreover, for any $\eta$ given by ($\mathcal{F}.2$) and any $g\in H$, it holds that $\tilde{g}=\eta(1,g)$ belongs to $H$, and so ($\mathcal{F}.2$) is satisfied. Since ($\mathcal{F}.1$) and ($\mathcal{F}.2$) are satisfied, Lemma \ref{Ambro Ruiz critical levels} applies. Thus, for any $m\in\N$, we denote by $\mathcal{J}_m$ the set of values $\mu\in \left[\frac{1}{2},1\right]$ such that the function $\mu\mapsto c_{m,\mu}$ is differentiable. We then let
\[\mathcal{M}\coloneqq \bigcap\limits_{m\in\N} \mathcal{J}_m.\]
We note that since 
\[\left[\frac{1}{2},1\right]\setminus \mathcal{M} = \bigcup\limits_{m\in\N} \left(\left[\frac{1}{2},1\right]\setminus \mathcal{J}_m\right)\]
and $[\frac{1}{2},1]\setminus\mathcal{J}_m$ has zero Lebesgue measure for each $m$ by Lemma \ref{Ambro Ruiz critical levels}, then it follows that $\left[\frac{1}{2},1\right]\setminus \mathcal{M}$ has zero Lebesgue measure. Arguing as in the proof of Corollary \ref{density corollary}, we obtain that $\mathcal{M}$ is dense in $\left[\frac{1}{2},1\right]$. We can now apply Proposition \ref{Ambro Ruiz bounded PS} with $\Phi_{\mu}=I_{\mu}$. Namely, for each fixed $m\in\N$ and $\mu\in\mathcal{M}$ we obtain that there exists a bounded sequence $(u_n)_{n\in\N}\subset E(\R^N)$ such that $I_{\mu}(u_n)\to c_{m,\mu}$ and $I'_{\mu}(u_n)\to 0$. The embedding of $E(\R^N)$ into $L^{q+1}(\R^N)$ is compact by Lemma \ref{compactembeddingofeintolpplus} so, arguing as in the proof of Theorem \ref{mountainpasssoltheoremlowp}, we can show that the values $c_{m, \mu}$ are critical levels of $I_{\mu}$ for each $m\in\N$ and $\mu\in\mathcal{M}$. We then take $m$ fixed, $(\mu_n)_{n\in\N}$ an increasing sequence in $\mathcal{M}$ such that $\mu_n\to 1$, and $(u_n)_{n\in\N}\subset E(\R^N)$ such that $I'_{\mu_n}(u_n)= 0$ and $I_{\mu_n}(u_n)=c_{m, \mu_n}$. We note that since $\rho$ is homogeneous of degree $\bar{k}$ by assumption, it follows from \cite[p. 296]{Gelfand and Shilov} that $\bar{k} \rho(x)=(x, \nabla \rho)$. So, setting  $\alpha_n = \intrn (\abs{\nabla \un}^2 + \un^2)$, $\gamma_n = \intrn \rho(x) \phi_{\un} \un^2$, $\delta_n = \mu_n \intrn \abs{\un}^{q+1}$ and using the Pohozaev-type condition deduced in Lemma \ref{pohozaevlemma}, we obtain the system
	\begin{equation}
	    \begin{cases}
            \begin{array}{c c c c c c c}
	            \alpha_n & + & \gamma_n & - & \delta_n & = & 0,\\
	            \half \alpha_n & + & \frac{1}{4} \gamma_n & - &
	            \overqplus \delta_n & = & c_{m,\mu_n},\\
	            \frac{N-2}{2} \alpha_n & + & \left( \frac{N+2+2k}{4} \right) \gamma_n & - & \frac{N}{q+1} \delta_n & \leq & 0.\\
	        \end{array}
	        \end{cases}
	        \end{equation}
Since the assumptions on $\bar{k}$ guarantee that $\bar{k} > \frac{-2(q-2)}{(q-1)}>\frac{N-6}{2}$ for $q\in(2,3]$ if $N=3$ and for $q\in(2,\tstar-1)$ if $N=4,5$, it follows that we can solve this system and show that $\alpha_n,\gamma_n,\delta_n$ are all bounded as in the proof of Theorem \ref{mountainpasssoltheoremlowp}. Moreover, continuing to argue as in the proof of this theorem and using the compact embedding of $E(\R^N)$ into $L^{q+1}(\R^N)$, we can then prove that for each fixed $m$ there exists $u\in E(\R^N)$ such that, up to a subsequence, $u_n\to u$ in $E(\R^N)$, $I(u)=I_1(u)=c_{m,1}$, and $I'(u)=I_1'(u)=0$.
It therefore remains to show that $I(u)=c_{m,1} \to+\infty$ as $m\to+\infty$. In order to do so, we define
\[\tilde{\Gamma}_m\coloneqq \left\{g\in C(E_m\cap B_1, E(\R^N)): g \text{ is odd, one-to-one, } I(g(y))\leq 0 \text{ for all } y\in \partial(E_m\cap B_1)\right\},\]
\[\tilde{G}_m\coloneqq \left\{A\subset E(\R^N): A=g(E_m\cap B_1), g\in\tilde{\Gamma}_m\right\},\]
\[\tilde{b}_m\coloneqq \inf_{A\in\tilde{G}_m}\max_{u\in A} I(u).\]
We then note that by \cite[Corollary $2.16$]{Ambrosetti and Rabinowitz}, it holds that 
\[d_m\leq \tilde{b}_m,\]
where $d_m$ is given by \eqref{def of dm}. It therefore follows from Lemma \ref{divergence of critical levels} that 
\begin{equation}\label{tilde bm diverge}
\tilde{b}_m\to+\infty, \text{ as } m\to+\infty.
\end{equation}
We will now show $G_m\subseteq \tilde{G}_m$. We take $A\in G_m$. Then, by definition, there exists $g\in H$ such that $A=g(A_m)$. We define an odd homeomorphism $\varphi:E_m\cap B_1\to A_m$ by
\[\varphi(e) = t^{\nu}u(t\cdot), \quad \text{with } t=T_m||e||_{E(\R^N)}, \,u=\frac{e}{||e||_{E(\R^N)}},\]
where $T_m$ is defined in \eqref{def of Tm}, and set $\tilde{g}=g\circ\varphi$. Since we can write $A=\tilde{g}(E_m\cap B_1)$, then by the definition of $\tilde{G}_m$ we need only to show that $\tilde{g}\in\tilde{\Gamma}_m$. Clearly, $\tilde{g}\in C(E_m\cap B_1, E(\R^N))$ is odd and one-to-one. Moreover, for every $y\in\partial(E_m\cap B_1)$, setting $w=\varphi(y)\in\partial A_m$, we have $I(\tilde{g}(y))=I(g(w))$. Since $g\in H$ and $w\in\partial A_m$, then by definition it holds that $g(w)=w$. Putting everything together, we have
\[I(\tilde{g}(y))=I(g(w))=I(w)\leq \sup_{w\in\partial A_m}I(w)\leq 0,\]
where the final inequality follows from \eqref{sup on partial Am less than 0}. Hence, we have shown $\tilde{g}\in\tilde{\Gamma}_m$ and so $G_m\subseteq \tilde{G}_m$. Therefore, for each $m\in\N$, it follows that 
\[\tilde{b}_m= \inf_{A\in\tilde{G}_m}\max_{u\in A} I(u)\leq  \inf_{A\in G_m}\max_{u\in A} I(u)=c_{m,1},\]
and so, by \eqref{tilde bm diverge}, we conclude that 
\[c_{m,1}\to+\infty, \text{ as } m\to+\infty,\]
as required.
\end{proof}

\section*{Appendix A: Proof of the Pohozaev-type condition  }

\begin{proof}[Proof of Lemma \ref{pohozaevlemma}]
With the regularity remarks of Proposition \ref{reg} in place, we now multiply the first equation in \eqref{poh system} by $(x, \nabla u)$ and integrate on $B_R(0)$ for some $R>0$. We will compute each integral separately. We first note that 
\begin{equation}\label{POH1}
\begin{split}
\int_{B_R}-\Delta u (x, \nabla u) \dif x =  \frac{2-N}{2}&\int_{B_R}|\nabla u|^2 \dif x \\
&-\frac{1}{R}\int_{\partial B_R} |(x,\nabla u)|^2 \dif  \sigma +\frac{R}{2}\int_{\partial B_R}|\nabla u|^2 \dif \sigma.
\end{split}
\end{equation}
\noindent Fixing $i =1,\dots, N$, integrating by parts and using the divergence theorem, we then see that,
\begin{align}
\int_{B_R} bu(x_i \partial_i u)\dif x& =b\left[ -\frac{1}{2} \int_{B_R} u^2\dif x +\frac{1}{2}\int_{B_R}\partial_i(u^2x_i)\dif x \right] \nonumber \\ 
&=b\left[ -\frac{1}{2} \int_{B_R} u^2\dif x + \frac{1}{2}\int_{\partial B_R} u^2 \frac{x_i^2}{|x|}\dif \sigma \right].
\nonumber
\end{align}
\noindent So, summing over $i$, we get
\begin{equation}\label{POH2}
\int_{B_R} bu(x, \nabla u) \dif x= b\left[-\frac{N}{2}\int_{B_R}u^2\dif x +\frac{R}{2}\int_{\partial B_R}u^2 \dif \sigma \right].
\end{equation}
\noindent Again, fixing $i =1,\dots, N$, integrating by parts and using the divergence theorem, we find that,
\begin{align}
\int_{B_R}c\rho \phi_u u x_i (\partial_i u) \dif x &= c\bigg[ -\frac{1}{2}\int_{B_R} \rho\phi_u u^2 \dif x -\frac{1}{2}\int_{B_R} \phi_u u^2x_i (\partial_i\rho) \dif x  \nonumber \\
& \qquad\qquad -\frac{1}{2}\int_{B_R} \rho u^2  x_i  (\partial_i \phi_u) \dif x+\frac{1}{2}\int_{B_R}\partial_i (\rho\phi_u  u^2 x_i)\dif x \bigg] \nonumber\\
&=c\bigg[ -\frac{1}{2}\int_{B_R}\rho\phi_u  u^2 \dif x -\frac{1}{2}\int_{B_R} \phi_u u^2x_i (\partial_i\rho) \dif x \nonumber \\
& \qquad\qquad  -\frac{1}{2}\int_{B_R} \rho u^2  x_i  (\partial_i \phi_u) \dif x+\frac{1}{2}\int_{\partial B_R}\rho \phi_u  u^2 \frac{x_i^2}{|x|}\dif \sigma \bigg]. \nonumber
\end{align}
\noindent Thus, summing over $i$, we get
\begin{align}\label{POH3}
\int_{B_R}c\rho\phi_u u (x, \nabla u) \dif x &= c\bigg[ -\frac{N}{2}\int_{B_R}\rho\phi_u  u^2 \dif x -\frac{1}{2}\int_{B_R} \phi_u u^2 (x, \nabla \rho) \dif x  \nonumber\\
& \qquad\qquad -\frac{1}{2}\int_{B_R} \rho u^2  (x, \nabla \phi_u) \dif x+\frac{R}{2}\int_{\partial B_R} \rho\phi_u  u^2 \dif \sigma \bigg].  
\end{align}
\noindent Finally, once more fixing $i=1,\dots, N$, integrating by parts and using the divergence theorem, we find that,
\begin{align*}
\int_{B_R}d |u|^{q-1}u (x_i \partial_i u) \dif x = d \left[ \frac{-1}{q+1} \int_{B_R} |u|^{q+1} \dif x+ \frac{1}{q+1} \int_{\partial B_R} |u|^{q+1} \frac{x_i^2}{|x|} \dif \sigma \right],
\end{align*}
\noindent and so, summing over $i$, we see that 
\begin{equation}\label{POH4}
\begin{split}
\int_{B_R}d|u|^{q-1}u (x, \nabla u) \dif x=d\bigg[\frac{-N}{q+1}&\int_{B_R}|u|^{q+1}\dif x\\
&+\frac{R}{q+1}\int_{\partial B_R}|u|^{q+1}\dif \sigma\bigg].
\end{split}
\end{equation}
\noindent Putting \eqref{POH1}, \eqref{POH2}, \eqref{POH3} and \eqref{POH4} together, we see that
\begin{equation}\label{prePOH}
\begin{split}
\frac{2-N}{2}&\int_{B_R}|\nabla u|^2 \dif x -\frac{1}{R}\int_{\partial B_R} |(x,\nabla u)|^2 \dif \sigma +\frac{R}{2}\int_{\partial B_R}|\nabla u|^2 \dif \sigma  \\
&+b\bigg[-\frac{N}{2}\int_{B_R}u^2\dif x  +\frac{R}{2}\int_{\partial B_R}u^2 \dif \sigma \bigg] \\
&\qquad+ c\bigg[ -\frac{N}{2}\int_{B_R}\rho\phi_u  u^2 \dif x-\frac{1}{2}\int_{B_R} \phi_u u^2 (x, \nabla \rho) \dif x\\
&\qquad\qquad-\frac{1}{2}\int_{B_R} \rho u^2  (x, \nabla \phi_u) \dif x +\frac{R}{2}\int_{\partial B_R} \rho \phi_u  u^2 \dif \sigma \bigg]\\
&\qquad\qquad\qquad-d\left[\frac{-N}{q+1}\int_{B_R}|u|^{q+1}\dif x+\frac{R}{q+1}\int_{\partial B_R}|u|^{q+1}\dif \sigma\right] =0.
\end{split}
\end{equation}
\noindent We now multiply the second equation in \eqref{poh system} by $(x, \nabla \phi_u)$ and integrate on $B_R(0)$ for some $R>0$. By a simple calculation we see that 
\begin{equation*}
\begin{split}
\int_{B_R} \rho u^2 (x, \nabla \phi_u)\dif x &= \int_{B_R} -\Delta \phi_u (x, \nabla \phi_u) \dif x \\
&= \frac{2-N}{2}\int_{B_R}|\nabla \phi_u|^2 \dif x -\frac{1}{R}\int_{\partial B_R} |(x,\nabla \phi_u)|^2 \dif  \sigma \\
&\qquad+\frac{R}{2}\int_{\partial B_R}|\nabla \phi_u|^2 \dif \sigma.
\end{split}
\end{equation*}
\noindent Substituting this into \eqref{prePOH} and rearranging, we get
\begin{equation}\label{POHwithboundary}
\begin{split}
& \frac{N-2}{2}\int_{B_R}|\nabla u|^2 \dif x +\frac{Nb}{2}\int_{B_R}u^2\dif x +\frac{(N+k)c}{2}\int_{B_R}\rho\phi_u  u^2 \dif x  \\
&\qquad\qquad   +\frac{c(2-N)}{4}\int_{B_R}|\nabla \phi_u|^2\dif x -\frac{Nd}{q+1}\int_{B_R}|u|^{q+1}\dif x \\
& \qquad\leq\frac{N-2}{2}\int_{B_R}|\nabla u|^2 \dif x +\frac{Nb}{2}\int_{B_R}u^2\dif x +\frac{Nc}{2}\int_{B_R}\rho\phi_u  u^2 \dif x  \\
&\qquad\qquad   +\frac{c}{2}\int_{B_R} \phi_u u^2 (x, \nabla \rho) \dif x +\frac{c(2-N)}{4}\int_{B_R}|\nabla \phi_u|^2\dif x-\frac{Nd}{q+1}\int_{B_R}|u|^{q+1}\dif x \\
&\qquad=-\frac{1}{R}\int_{\partial B_R} |(x,\nabla u)|^2 \dif \sigma +\frac{R}{2}\int_{\partial B_R}|\nabla u|^2 \dif \sigma +\frac{bR}{2}\int_{\partial B_R}u^2 \dif\sigma   \\
&\qquad\qquad +\frac{cR}{2}\int_{\partial B_R} \rho\phi_u  u^2 \dif \sigma +\frac{c}{2R}\int_{\partial B_R} |(x,\nabla \phi_u)|^2 \dif \sigma \\
&\qquad\qquad\qquad-\frac{cR}{4}\int_{\partial B_R}|\nabla \phi_u|^2 \dif \sigma-\frac{dR}{q+1}\int_{\partial B_R}|u|^{q+1}\dif \sigma,
\end{split}
\end{equation}
\noindent where we have used the assumption $k\rho(x)\leq (x,\nabla \rho)$ for some $k\in\R$ to obtain the first inequality. We now call the right hand side of \eqref{POHwithboundary} $I_R$, namely
\begin{equation*}
\begin{split}
I_R & \coloneqq-\frac{1}{R}\int_{\partial B_R} |(x,\nabla u)|^2 \dif \sigma +\frac{R}{2}\int_{\partial B_R}|\nabla u|^2 \dif \sigma +\frac{bR}{2}\int_{\partial B_R}u^2 \dif\sigma   \\
&\qquad\qquad +\frac{cR}{2}\int_{\partial B_R} \rho\phi_u  u^2 \dif \sigma +\frac{c}{2R}\int_{\partial B_R} |(x,\nabla \phi_u)|^2 \dif \sigma \\
&\qquad\qquad\qquad-\frac{cR}{4}\int_{\partial B_R}|\nabla \phi_u|^2 \dif \sigma-\frac{dR}{q+1}\int_{\partial B_R}|u|^{q+1}\dif \sigma.
\end{split}
\end{equation*}
We note that $|(x,\nabla u)|\leq R|\nabla u |$ and $|(x,\nabla \phi_u)|\leq R|\nabla \phi_u |$ on $\partial B_R$, so it holds that
\begin{align*}
|I_R| &\leq \frac{3R}{2}\int_{\partial B_R}|\nabla u|^2 \dif \sigma +\frac{bR}{2}\int_{\partial B_R}u^2 \dif\sigma  \\
&\qquad + \frac{cR}{2}\int_{\partial B_R}\rho \phi_u  u^2 \dif \sigma+\frac{3cR}{4}\int_{\partial B_R}|\nabla \phi_u|^2 \dif \sigma+\frac{dR}{q+1}\int_{\partial B_R}|u|^{q+1}\dif\sigma.
\end{align*}
\noindent Now, since $|\nabla u|^2$, $u^2 \in L^1(\R^N)$ as $u\in E (\R^N)\subseteq H^1(\R^N)$, $ \rho\phi_u u^2$, $|\nabla \phi_u|^2 \in L^1(\R^N)$ because $\int_{\R^N}\rho\phi_u u^2\dif x=\int_{\R^N}|\nabla \phi_u|^2\dif x$ and $\phi_u \in D^{1,2}(\R^N)$, and $|u|^{q+1} \in L^1(\R^N)$ because $E(\R^N) \hookrightarrow L^s(\R^N)$ for all $s \in [2,\tstar]$, then it holds that $I_{R_n} \to 0$ as $n\to +\infty$ for a suitable sequence $R_n \to +\infty.$  Moreover, since \eqref{POHwithboundary} holds for any $R>0$, it follows that 
\begin{align*}\label{POH limit}
\frac{N-2}{2}\int_{\R^N}|\nabla u|^2 \dif x +\frac{Nb}{2}&\int_{\R^N}u^2\dif x  +\frac{(N+k)c}{2}\int_{\R^N}\rho\phi_u  u^2 \dif x \\
&+\frac{c(2-N)}{4}\int_{\R^N}|\nabla \phi_u|^2\dif x -\frac{Nd}{q+1}\int_{\R^N}|u|^{q+1}\dif x \leq 0,
\end{align*}
and so, we obtain
\[\frac{N-2}{2}\int_{\R^N}|\nabla u|^2 \dif x +\frac{Nb}{2}\int_{\R^N}u^2\dif x  +\frac{(N+2+2k)c}{4}\int_{\R^N}\rho\phi_u  u^2 \dif x-\frac{Nd}{q+1}\int_{\R^N}|u|^{q+1}\dif x \leq 0, \]\\
using the fact that $\int_{\R^N}|\nabla \phi_u|^2\dif x=\int_{\R^N}\rho\phi_u u^2\dif x $. This completes the proof.
\end{proof}
\section*{Data availability statement}
On behalf of all authors, the corresponding author states that there are no data associated to our manuscripts.

\section*{Conflicts of interest/Competing interests}
On behalf of all authors, the corresponding author states that there is no conflict of interest.

\renewcommand\refname{References}

\end{document}